\documentclass{article}
\usepackage[utf8]{inputenc}
\usepackage{authblk}

\usepackage{latexsym,amsthm,amscd}
\usepackage{amsmath}
\usepackage{amssymb}
\usepackage{amsfonts}
\usepackage[normalem]{ulem}
\usepackage{tikz}
\usepackage{stmaryrd}
\usepackage{hyperref}
\usepackage[most]{tcolorbox}
\usetikzlibrary{patterns}
\usepackage{enumerate}
\usepackage{units}
\usepackage{CJKutf8}
\usepackage{ulem}
\usepackage{geometry}

\usepackage{mathrsfs}

\usepackage{float}
\usepackage{caption}
\usepackage{subcaption}

\usepackage{soul,xcolor}
\usepackage[capitalize]{cleveref}

\newtheorem{thmx}{Theorem}

\newcommand{\A}{\mathcal{A}}
\newcommand{\Z}{\mathbb{Z}}
\newcommand{\N}{\mathbb{N}}
\newcommand{\ie}{\textit{i.e.}}
\newcommand{\eg}{\textit{e.g.}}

\title{The topological structure of isolated points in the space of $\mathbb{Z}^d$-shifts}
\author[1]{\normalsize SILV\`ERE GANGLOFF\footnote{Corresponding author}\footnote {\url{sfgangloff@gmail.com}}} 

\affil[1]{\small
	Independent researcher
	}

\author[2,3]{\normalsize ALONSO N\'U\~NEZ\footnote{The second author is supported by the National Agency for Research and Development
(ANID) / Scholarship Program / DOCTORADO BECAS CHILE/2019 - 72200562; by the ANR “Difference” project (ANR-20-CE40-0002); and by the CIMI LabEx “Computability of asymptotic properties of dynamical systems” project (ANR-20-CE48-0002).}\footnote{\url{alonso.herrera_nunez@math.univ-toulouse.fr}}}
\affil[2]{
\small
Institut de mathématiques de Toulouse, 1R3, Université Paul Sabatier, 118 Route de Narbonne, 31400 Toulouse}
\affil[3]{Instituto de Ingenier\'ia Matem\'atica y Computacional, Pontificia Universidad Cat\'olica de Chile, Chile
}
\date{\today}

\newtheorem{theorem}{Theorem}[section]
\newtheorem*{theorem*}{Theorem}
\newtheorem{lemma}[theorem]{Lemma}
\newtheorem{corollary}[theorem]{Corollary}
\newtheorem{question}[theorem]{Question}
\newtheorem{definition}[theorem]{Definition}
\newtheorem{example}[theorem]{Example}
\newtheorem{conjecture}[theorem]{Conjecture}
\newtheorem*{problem*}{Problem}
\newtheorem{notation}[theorem]{Notation}
\newtheorem{proposition}[theorem]{Proposition}
\newtheorem{remark}[theorem]{Remark}

\newcommand{\robionegauche}[2]{
\draw (#1,#2) rectangle (#1+2,#2+2);
\draw [-latex] (#1+2,#2+1) -- (#1+0,#2+1);
\draw [-latex] (#1+1,#2+0) -- (#1+1,#2+1); 
\draw [-latex] (#1+1,#2+2) -- (#1+1,#2+1);}

\newcommand{\robionehaut}[2]{
\draw (#1,#2) rectangle (#1+2,#2+2);
\draw [-latex] (#1+1,#2+2) -- (#1+1,#2+0);
\draw [-latex] (#1+0,#2+1) -- (#1+1,#2+1); 
\draw [-latex] (#1+2,#2+1) -- (#1+1,#2+1);}

\newcommand{\robitwobas}[2]{
\draw (#1,#2) rectangle (#1+2,#2+2);
\draw [-latex] (#1+1,#2+2) -- (#1+1,#2+0) ;
\draw [-latex] (#1+0,#2+1) -- (#1+1,#2+1) ; 
\draw [-latex] (#1+0,#2+0.5) -- (#1+1,#2+0.5) ; 
\draw [-latex] (#1+2,#2+1) -- (#1+1,#2+1) ;
\draw [-latex] (#1+2,#2+0.5) -- (#1+1,#2+0.5) ;}
\newcommand{\robitwohaut}[2]{
\draw (#1,#2) rectangle (#1+2,#2+2);
\draw [-latex] (#1+1,#2+0) -- (#1+1,#2+2) ;
\draw [-latex] (#1+0,#2+1) -- (#1+1,#2+1) ; 
\draw [-latex] (#1+0,#2+1.5) -- (#1+1,#2+1.5) ; 
\draw [-latex] (#1+2,#2+1) -- (#1+1,#2+1) ;
\draw [-latex] (#1+2,#2+1.5) -- (#1+1,#2+1.5) ;}
\newcommand{\robitwodroite}[2]{
\draw (#1,#2) rectangle (#1+2,#2+2);
\draw [-latex] (#1+0,#2+1) -- (#1+2,#2+1) ;
\draw [-latex] (#1+1,#2+0) -- (#1+1,#2+1) ; 
\draw [-latex] (#1+1.5,#2+0) -- (#1+1.5,#2+1) ; 
\draw [-latex] (#1+1,#2+2) -- (#1+1,#2+1) ;
\draw [-latex] (#1+1.5,#2+2) -- (#1+1.5,#2+1) ;}
\newcommand{\robitwogauche}[2]{
\draw (#1,#2) rectangle (#1+2,#2+2);
\draw [-latex] (#1+2,#2+1) -- (#1+0,#2+1) ;
\draw [-latex] (#1+1,#2+0) -- (#1+1,#2+1) ; 
\draw [-latex] (#1+0.5,#2+0) -- (#1+0.5,#2+1) ; 
\draw [-latex] (#1+1,#2+2) -- (#1+1,#2+1) ;
\draw [-latex] (#1+0.5,#2+2) -- (#1+0.5,#2+1) ;}

\newcommand{\robithreehaut}[2]{
\draw (#1,#2) rectangle (#1+2,#2+2) ;
\draw [-latex] (#1+1,#2+0) -- (#1+1,#2+2) ;
\draw [-latex] (#1+0.5,#2+0) -- (#1+0.5,#2+2) ; 
\draw [-latex] (#1+0,#2+1) -- (#1+0.5,#2+1) ; 
\draw [-latex] (#1+2,#2+1) -- (#1+1,#2+1) ;}

\newcommand{\robithreedroite}[2]{
\draw (#1,#2) rectangle (#1+2,#2+2) ;
\draw [-latex] (#1+0,#2+1) -- (#1+2,#2+1) ;
\draw [-latex] (#1+0,#2+0.5) -- (#1+2,#2+0.5) ; 
\draw [-latex] (#1+1,#2+0) -- (#1+1,#2+0.5) ; 
\draw [-latex] (#1+1,#2+2) -- (#1+1,#2+1) ;}

\newcommand{\robisixbas}[2]{
\draw (#1,#2) rectangle (#1+2,#2+2) ;
\draw [-latex] (#1+1,#2+2) -- (#1+1,#2+0) ;
\draw [-latex] (#1+0.5,#2+2) -- (#1+0.5,#2+0) ; 
\draw [-latex] (#1+0,#2+1) -- (#1+0.5,#2+1) ; 
\draw [-latex] (#1+2,#2+1) -- (#1+1,#2+1) ;
\draw [-latex] (#1+0,#2+0.5) -- (#1+0.5,#2+0.5) ; 
\draw [-latex] (#1+2,#2+0.5) -- (#1+1,#2+0.5) ;}
\newcommand{\robisixgauche}[2]{
\draw (#1,#2) rectangle (#1+2,#2+2) ;
\draw [-latex] (#1+2,#2+1) -- (#1+0,#2+1) ;
\draw [-latex] (#1+2,#2+0.5) -- (#1+0,#2+0.5) ; 
\draw [-latex] (#1+1,#2+0) -- (#1+1,#2+0.5) ; 
\draw [-latex] (#1+1,#2+2) -- (#1+1,#2+1) ;
\draw [-latex] (#1+0.5,#2+0) -- (#1+0.5,#2+0.5) ; 
\draw [-latex] (#1+0.5,#2+2) -- (#1+0.5,#2+1) ;}
\newcommand{\robisixhaut}[2]{
\draw (#1,#2) rectangle (#1+2,#2+2) ;
\draw [-latex] (#1+1,#2+0) -- (#1+1,#2+2) ;
\draw [-latex] (#1+0.5,#2+0) -- (#1+0.5,#2+2) ; 
\draw [-latex] (#1+0,#2+1) -- (#1+0.5,#2+1) ; 
\draw [-latex] (#1+2,#2+1) -- (#1+1,#2+1) ;
\draw [-latex] (#1+0,#2+1.5) -- (#1+0.5,#2+1.5) ; 
\draw [-latex] (#1+2,#2+1.5) -- (#1+1,#2+1.5) ;}
\newcommand{\robisixdroite}[2]{
\draw (#1,#2) rectangle (#1+2,#2+2) ;
\draw [-latex] (#1+0,#2+1) -- (#1+2,#2+1) ;
\draw [-latex] (#1+0,#2+0.5) -- (#1+2,#2+0.5) ; 
\draw [-latex] (#1+1,#2+0) -- (#1+1,#2+0.5) ; 
\draw [-latex] (#1+1,#2+2) -- (#1+1,#2+1) ;
\draw [-latex] (#1+1.5,#2+0) -- (#1+1.5,#2+0.5) ; 
\draw [-latex] (#1+1.5,#2+2) -- (#1+1.5,#2+1) ;}

\newcommand{\robisevenbas}[2]{
\draw (#1,#2) rectangle (#1+2,#2+2) ;
\draw [-latex] (#1+1,#2+2) -- (#1+1,#2+0) ;
\draw [-latex] (#1+1.5,#2+2) -- (#1+1.5,#2+0) ; 
\draw [-latex] (#1+0,#2+1) -- (#1+1,#2+1) ; 
\draw [-latex] (#1+2,#2+1) -- (#1+1.5,#2+1) ;
\draw [-latex] (#1+0,#2+0.5) -- (#1+1,#2+0.5) ; 
\draw [-latex] (#1+2,#2+0.5) -- (#1+1.5,#2+0.5) ;}
\newcommand{\robisevengauche}[2]{
\draw (#1,#2) rectangle (#1+2,#2+2) ;
\draw [-latex] (#1+2,#2+1) -- (#1+0,#2+1) ;
\draw [-latex] (#1+2,#2+1.5) -- (#1+0,#2+1.5) ; 
\draw [-latex] (#1+1,#2+0) -- (#1+1,#2+1) ; 
\draw [-latex] (#1+1,#2+2) -- (#1+1,#2+1.5) ;
\draw [-latex] (#1+0.5,#2+0) -- (#1+0.5,#2+1) ; 
\draw [-latex] (#1+0.5,#2+2) -- (#1+0.5,#2+1.5) ;}
\newcommand{\robisevenhaut}[2]{
\draw (#1,#2) rectangle (#1+2,#2+2) ;
\draw [-latex] (#1+1,#2+0) -- (#1+1,#2+2) ;
\draw [-latex] (#1+1.5,#2+0) -- (#1+1.5,#2+2) ; 
\draw [-latex] (#1+0,#2+1) -- (#1+1,#2+1) ; 
\draw [-latex] (#1+2,#2+1) -- (#1+1.5,#2+1) ;
\draw [-latex] (#1+0,#2+1.5) -- (#1+1,#2+1.5) ; 
\draw [-latex] (#1+2,#2+1.5) -- (#1+1.5,#2+1.5) ;}
\newcommand{\robisevendroite}[2]{
\draw (#1,#2) rectangle (#1+2,#2+2) ;
\draw [-latex] (#1+0,#2+1) -- (#1+2,#2+1) ;
\draw [-latex] (#1+0,#2+1.5) -- (#1+2,#2+1.5) ; 
\draw [-latex] (#1+1,#2+0) -- (#1+1,#2+1) ; 
\draw [-latex] (#1+1,#2+2) -- (#1+1,#2+1.5) ;
\draw [-latex] (#1+1.5,#2+0) -- (#1+1.5,#2+1) ; 
\draw [-latex] (#1+1.5,#2+2) -- (#1+1.5,#2+1.5) ;}

\newcommand{\robibluebastgauche}[2]{
\fill[blue!40] (#1+0.5,#2+0.5) rectangle (#1+1,#2+2) ;
\fill[blue!40] (#1+0.5,#2+0.5) rectangle (#1+2,#2+1) ;
\draw (#1,#2) rectangle (#1+2,#2+2) ;
\draw [-latex] (#1+0.5,#2+0.5) -- (#1+0.5,#2+2) ; 
\draw [-latex] (#1+0.5,#2+0.5) -- (#1+2,#2+0.5) ;
\draw [-latex] (#1+1,#2+1) -- (#1+1,#2+2) ; 
\draw [-latex] (#1+1,#2+1) -- (#1+2,#2+1) ; 
\draw [-latex] (#1+0.5,#2+1) -- (#1+0,#2+1) ; 
\draw [-latex] (#1+1,#2+0.5) -- (#1+1,#2+0) ;
\node[scale=0.85] at (#1+0.825,#2+0.825) {\textbf{0}};
}
\newcommand{\robibluebastdroite}[2]{
\fill[blue!40] (#1+1.5,#2+0.5) rectangle (#1+1,#2+2) ;
\fill[blue!40] (#1+1.5,#2+0.5) rectangle (#1+0,#2+1) ;
\draw (#1,#2) rectangle (#1+2,#2+2) ;
\draw [-latex] (#1+1.5,#2+0.5) -- (#1+1.5,#2+2) ; 
\draw [-latex] (#1+1.5,#2+0.5) -- (#1+0,#2+0.5) ;
\draw [-latex] (#1+1,#2+1) -- (#1+1,#2+2) ; 
\draw [-latex] (#1+1,#2+1) -- (#1+0,#2+1) ; 
\draw [-latex] (#1+1.5,#2+1) -- (#1+2,#2+1) ; 
\draw [-latex] (#1+1,#2+0.5) -- (#1+1,#2+0) ;
\node[scale=0.85] at (#1+1.175,#2+0.825) {\textbf{0}};
}
\newcommand{\robibluehauttgauche}[2]{
\fill[blue!40] (#1+2,#2+1) rectangle (#1+0.5,#2+1.5) ;
\fill[blue!40] (#1+1,#2+0) rectangle (#1+0.5,#2+1.5) ;
\draw (#1,#2) rectangle (#1+2,#2+2) ;
\draw [-latex] (#1+0.5,#2+1.5) -- (#1+0.5,#2+0) ; 
\draw [-latex] (#1+0.5,#2+1.5) -- (#1+2,#2+1.5) ;
\draw [-latex] (#1+1,#2+1) -- (#1+1,#2+0) ; 
\draw [-latex] (#1+1,#2+1) -- (#1+2,#2+1) ; 
\draw [-latex] (#1+0.5,#2+1) -- (#1+0,#2+1) ; 
\draw [-latex] (#1+1,#2+1.5) -- (#1+1,#2+2) ;
\node[scale=0.85] at (#1+0.825,#2+1.175) {\textbf{0}};
}
\newcommand{\robibluehauttdroite}[2]{
\fill[blue!40] (#1+0,#2+1) rectangle (#1+1.5,#2+1.5) ;
\fill[blue!40] (#1+1,#2+0) rectangle (#1+1.5,#2+1.5) ;
\draw (#1,#2) rectangle (#1+2,#2+2) ;
\draw [-latex] (#1+1.5,#2+1.5) -- (#1+1.5,#2+0) ; 
\draw [-latex] (#1+1.5,#2+1.5) -- (#1+0,#2+1.5) ;
\draw [-latex] (#1+1,#2+1) -- (#1+1,#2+0) ; 
\draw [-latex] (#1+1,#2+1) -- (#1+0,#2+1) ; 
\draw [-latex] (#1+1.5,#2+1) -- (#1+2,#2+1) ; 
\draw [-latex] (#1+1,#2+1.5) -- (#1+1,#2+2) ;
\node[scale=0.85] at (#1+1.175,#2+1.175) {\textbf{0}};
}

\newcommand{\robiredbasgauche}[2]{
\fill[red!40] (#1+0.5,#2+0.5) rectangle (#1+1,#2+2) ;
\fill[red!40] (#1+0.5,#2+0.5) rectangle (#1+2,#2+1) ;
\draw (#1,#2) rectangle (#1+2,#2+2) ;
\draw [-latex] (#1+0.5,#2+0.5) -- (#1+0.5,#2+2) ; 
\draw [-latex] (#1+0.5,#2+0.5) -- (#1+2,#2+0.5) ;
\draw [-latex] (#1+1,#2+1) -- (#1+1,#2+2) ; 
\draw [-latex] (#1+1,#2+1) -- (#1+2,#2+1) ; 
\draw [-latex] (#1+0.5,#2+1) -- (#1+0,#2+1) ; 
\draw [-latex] (#1+1,#2+0.5) -- (#1+1,#2+0) ;
\node[scale=0.85] at (#1+0.825,#2+0.825) {\textbf{1}};
}
\newcommand{\robiredbasdroite}[2]{
\fill[red!40] (#1+1.5,#2+0.5) rectangle (#1+1,#2+2) ;
\fill[red!40] (#1+1.5,#2+0.5) rectangle (#1+0,#2+1) ;
\draw (#1,#2) rectangle (#1+2,#2+2) ;
\draw [-latex] (#1+1.5,#2+0.5) -- (#1+1.5,#2+2) ; 
\draw [-latex] (#1+1.5,#2+0.5) -- (#1+0,#2+0.5) ;
\draw [-latex] (#1+1,#2+1) -- (#1+1,#2+2) ; 
\draw [-latex] (#1+1,#2+1) -- (#1+0,#2+1) ; 
\draw [-latex] (#1+1.5,#2+1) -- (#1+2,#2+1) ; 
\draw [-latex] (#1+1,#2+0.5) -- (#1+1,#2+0) ;
\node[scale=0.85] at (#1+1.175,#2+0.825) {\textbf{1}};
}
\newcommand{\robiredhautgauche}[2]{
\fill[red!40] (#1+2,#2+1) rectangle (#1+0.5,#2+1.5) ;
\fill[red!40] (#1+1,#2+0) rectangle (#1+0.5,#2+1.5) ;
\draw (#1,#2) rectangle (#1+2,#2+2) ;
\draw [-latex] (#1+0.5,#2+1.5) -- (#1+0.5,#2+0) ; 
\draw [-latex] (#1+0.5,#2+1.5) -- (#1+2,#2+1.5) ;
\draw [-latex] (#1+1,#2+1) -- (#1+1,#2+0) ; 
\draw [-latex] (#1+1,#2+1) -- (#1+2,#2+1) ; 
\draw [-latex] (#1+0.5,#2+1) -- (#1+0,#2+1) ; 
\draw [-latex] (#1+1,#2+1.5) -- (#1+1,#2+2) ;
\node[scale=0.85] at (#1+0.825,#2+1.175) {\textbf{1}};
}
\newcommand{\robiredhautdroite}[2]{
\fill[red!40] (#1+0,#2+1) rectangle (#1+1.5,#2+1.5) ;
\fill[red!40] (#1+1,#2+0) rectangle (#1+1.5,#2+1.5) ;
\draw (#1,#2) rectangle (#1+2,#2+2) ;
\draw [-latex] (#1+1.5,#2+1.5) -- (#1+1.5,#2+0) ; 
\draw [-latex] (#1+1.5,#2+1.5) -- (#1+0,#2+1.5) ;
\draw [-latex] (#1+1,#2+1) -- (#1+1,#2+0) ; 
\draw [-latex] (#1+1,#2+1) -- (#1+0,#2+1) ; 
\draw [-latex] (#1+1.5,#2+1) -- (#1+2,#2+1) ; 
\draw [-latex] (#1+1,#2+1.5) -- (#1+1,#2+2) ;
\node[scale=0.85] at (#1+1.175,#2+1.175) {\textbf{1}};
}

\tikzset{every loop/.style={min distance=2cm}}

\begin{document}

\maketitle
\setstcolor{red}

\begin{abstract}
    R. Pavlov and S. Schmieding~\cite{Pavlov} provided recently some results about generic $\mathbb{Z}$-shifts, which rely mainly on an original theorem stating that isolated points form a residual set in the space of $\mathbb{Z}$-shifts such that all other residual sets must contain it. As a direction for further research, they pointed towards genericity in the space of $\mathbb{G}$-shifts, where $\mathbb{G}$ is a finitely generated group. In the present text, we approach this for the case of $\mathbb{Z}^d$-shifts, where $d \ge 2$.
    As it is usual, multidimensional dynamical systems are much more difficult to understand. In light of the result of R. Pavlov and S. Schmieding, it is natural to begin with a better understanding of isolated points. We prove here a characterization of such points in the space of $\mathbb{Z}^d$-shifts, in terms of the natural notion of maximal subsystems which we also introduce in this article. From this characterization we recover the result of R. Pavlov and S. Schmieding's for $\mathbb{Z}^1$-shifts. We also prove a series of results which exploit this notion. In particular some transitivity-like properties can be related to the number of maximal subsystems. Furthermore, we show that the Cantor-Bendixon rank of the space of $\mathbb{Z}^d$-shifts is infinite for $d>1$, while for $d=1$ it is known to be equal to one.
    % \footnote{\textcolor{red}{J'ai effac\'e la partie ``residuel'' car c'est d\'ej\`a montr\'e par Doucha et on ne donne que des exemples.}}}
    % \textcolor{gray}{Furthermore we recover that, on the contrary of dimension one, the set of isolated shifts is not residual and provide an example of shift which is not in its closure.}
    
\end{abstract}

%%%%%%%%%%%%%%%%%%%%%%%%%%%%%%%%%%%%%%%%%
%%%%%%%%%%%%%%%%%%%%%%%%%%%%%%%%%%%%%%%%%
% \paragraph*{Proposed intro}
\vspace{1cm}
Given a (large) collection of objects, it is somewhat natural to wonder which members within this family are typical. %Phrased like that, of course, the term `typical' does not carry any specific meaning and it might be interpreted through a plethora of viewpoints. %A way of making this question precise is formalizing `typical' through the notion of generic property: a property is called generic if the set of elements that satisfies it is residual (it contains a $G_\delta$ dense set). 
One way to formalize the idea of a typical object, in a topological sense, is the notion of dense $G_\delta$ set: a countable intersection of open dense sets. Properties satisfied by every element in a dense $G_\delta$ set are called \textit{generic properties}. Recall that a set is \textit{residual} if it contains a dense $G_\delta$ set.
%\comm{Silvere}{I think we can make the above a bit more to the point by saying "One way to formalise the idea of a typical object in a collection is the notion of $G_{\delta}$ dense set. Properties of typical objects for this definition are called generic."}
%\comm{Alonso}{You mean from ``Phrased like that..."? I'm ok with that, or from some other point, feel free to change it at your will :)}
In a dynamical systems context, where the collection of objects is a family of (dynamical) systems, the study of generic properties has taken many forms over the years, and it has produced a wide range of results.  The first results in this direction appeared as early as the 1940's with the works of Oxtoby and Ulam \cite{oxtoby}, asserting that a generic volume preserving homeomorphism of a cube is ergodic. Later Halmos \cite{halmos,halmos2,halmos3}, in 1944, showed that a generic, measure preserving transformation is weakly mixing; and four years later Rohlin \cite{rohlin} proved that such a transformation is not strongly mixing -- by showing that strongly mixing transformations are not generic. These two latter results combined mean (indirectly) that there exist weakly mixing systems that are not strongly mixing, constituting the first evidence of the existence of such systems. Multiple other results have been found, \eg, \cite{bezugly,glasner,akinhurley}

This kind of questions has been studied in the space of transformations over spaces conjugated to Cantor sets as well. For instance,
% \textcolor{gray}{When the ambient space is conjugated to a Cantor set, generic properties in the space of the associated transformations has not been the exception to this trend. }
Kechris and Rosendal \cite{Kechris} found in 2007 that the homeomorphism group of the Cantor space has but one generic conjugacy class. This result was quite surprising, since the variety of dynamics on the Cantor space is large and, in comparison, this generic class is rather `small'. The implication is that generic properties of Cantor systems are fairly degenerated. In 2008 a description of this system was provided by Akin, Glasner and Weiss \cite{Akin}. About the same time, 2008, Hochman \cite{Hochman} established a series of results for transitive and totally transitive homeomorphisms of the Cantor space. In particular he showed that in this space, generic, totally transitive transformations have zero topological entropy.

In this document we are concerned with the space $\mathcal{S}^d$ of all $\Z^d$-shifts, that is, the union of all shift spaces $\A^{\Z^d}$, with the union running through all finite alphabets $\A \subset \mathbb{Z}$.
% \footnote{\textcolor{red}{Ici le reviewer demande s'on veut vraiment prendre l'union sur tous les alphabets possibles, Pavlov prend que des sous-ensembles de $\mathbb{Z}$. \`A mon avis c'est la m\^eme chose, Pavlov utilise $\mathcal{A}\subsetneq\mathbb{Z}$ car on peut toujours re-nommer les symboles.} \textcolor{blue}{I think it might be dangerous indeed to take all the finite sets (although I am not really sure it matters). 
% We can do like Pavlov and say that for the examples we can recode symbols with integers without changing the properties}}. 
This space is typically endowed with the Hausdorff metric, and we keep this tradition in here. For $d=1$, plenty of genericity results have been obtained. In 1971, Sigmund showed that the set of subshifts with zero topological entropy is residual in the space $\mathcal{A}^\Z$, \cite{Sigmund}. Sigmund's result was extended by Frisch and Tamuz  to shifts spaces $\A^G$, \cite{Tamuz}, where $G$ is an amenable group. Moreover, in this setting, they showed that for any $c$, shifts with entropy $c$ are generic in the space of shifts with entropy at least $c$.

In a recent development, Pavlov and Schmieding \cite{Pavlov} showed in 2023 that in $S^1$ generic shifts are precisely the ones whose descriptive graphs satisfy a simple combinatorial property: no cycle in the graph has both incoming and outgoing edges. This description implies that a generic shift in $S^1$ is a countable shift of finite type that is the union of finitely many orbits which are bi-asymptotic to periodic orbits.  The consequences of this characterisation are plenty, notably, all known results about the structure of generic shifts can be recovered from this characterization. In particular, they easily recover Sigmund's result: generically, a shift has zero entropy.

As often occurs, the situation changes dramatically when considering multidimensional dynamics. In the case of shift spaces, see for instance \cite{callard,berger,drshen,meyerovitch,Robinson}. Thus, it is not surprising that a characterization of generic shifts in $\mathcal{S}^d$, $d\geq2$, becomes a heavy task. In \cite{Pavlov} the authors left open the question of generic shifts in the space of shifts on groups different from $\Z$. Observe that understanding isolated points (if there are any) is a reasonable first approach to generic properties, as every residual set must contain them. Our present work addresses this problem and provides a characterization of such points by the means of the notions of \textit{maximal subsystems} and \textit{outcasts}: i) a (non-empty) subsystem is maximal if it is a proper subsystem and it is not contained in any other proper subsystem; ii) a (non-empty) subsystem is an outcast if it is a proper subsystem and it is not contained in a maximal subsystem.

\begin{thmx}\label{thm:A}
A $\Z^d$-shift is isolated in $\mathcal{S}^d$ if and only if it is of finite type, it has finitely many maximal subsystems, and contains no outcast.
\end{thmx}

This characterization suggests the existence of isolated points with infinitely many subsystems, which does not occur in $S^1$. In \cref{section.example.inf.subsystems} we exhibit such a system using the well-known Robinson shift \cite{Robinson}. As the construction is quite involved and the example is not essential for the remainder of the paper, we do not provide all the technical details. We address this example to specialists and we leave the interested reader to find details, for instance, in \cite{Robinson,Gangloff}. The characterization in \cref{thm:A}  allows to recover Pavlov and Schmieding's, \cref{cor pavlov}. We further explore structural implications of the conditions on maximal subsystems through what we call \textit{maximality type}, \ie, a shift has maximality type $n\in\N_0\cup\{\infty\}$ if it contains exactly $n$ maximal subsystems. We use this notion to follow Hochman, Pavlov and Schmieding, and others, to study isolated points in the space of transitive shifts -- in $S^1$ those points do not even exist. In particular we found that transitive shifts have maximality type either 0 or 1. Furthermore, we characterize isolated points in this space:

\begin{thmx}\label{thm:B}
    In the space $\mathcal{T}^d$ of transitive $\mathbb{Z}^d$-shifts, a shift of finite type is isolated if it has maximality type 1, or if it is minimal.
    Furthermore, a non finite type shift $X$ with maximality type 1 is isolated if and only if it is contained in a shift of finite type $Z$ such that $Z \backslash X$ is not dense in $Z$.
\end{thmx}

We also provide, in \cref{non SFT isolated in T}, an example of a transitive shift which is not of finite type and is isolated. This implies in particular that minimality is not a generic property, contrary to the one-dimensional case. We find as well that strong mixing-type properties, such as block gluing, imply that the shift has maximality type 0. \bigskip 

These results exhibit a tight relation between maximal subsystems and the topological structure of shift spaces, and genericity might be understood in these terms. Note that although we do not consider other groups than $\mathbb{Z}^d$, we expect our result to be generalisable to other groups without particular difficulty. \bigskip 

We identify two approaches to the study of genericity: 
\begin{enumerate}
\item[\textbf{1:}] Computing the closure of the set of isolated points, which is challenging. However, we provide here examples of shifts that are not in this closure, and examples of non-isolated shifts that are in it; 
\item[\textbf{2:}] Computing the Cantor-Bendixon rank of the space of $\mathbb{Z}^d$-shifts in order to shed some additional light into the topological structure of this space. Here we show that the rank is infinite, but a precise computation remains an open question.
\end{enumerate}
% \textcolor{gray}{From this point, the strategy for understanding genericity is twofold: \textbf{1}. one way would be to compute the closure of the set of isolated shifts. This question is quite hard, and we only provide here examples of shifts which are not in this closure and non isolated shifts which are in the closure; \textbf{2}. Another way is to understand more generally the topological structure of the space of $\mathbb{Z}^d$-shifts, by computing its Cantor-Bendixson rank.} \bigskip

In this second direction, using properties of the well-known $\times 2\times3$ system on the circle, we build examples of shifts proving that:

\begin{thmx}\label{thm:C}
For all $d > 1$, the Cantor-Bendixson rank of the space of $\mathbb{Z}^d$-shifts is infinite.
\end{thmx}

In direction \textbf{1}, the strategy would be to better understand the structure of subsystems of isolated shifts, and more generally how the maximality type affects this structure. We may find some rigidity there which must be satisfied also by shifts in the closure of the set of isolated points. \bigskip

The text is organized as follows. Section~\ref{section.background} contains some background definitions and elements of constructions for some examples presented after. In Section~\ref{section.isolated.shifts} we prove the characterization of isolated points in the space of $\mathbb{Z}^d$-shifts. Section~\ref{section.transitive} is devoted to transitive shifts. In Section~\ref{section.dec} we prove that every shift can be written as a union of a sequence of shifts of maximality type $1$ and a shift of maximality type $0$. We also prove some properties of this decomposition which will be useful in the following section. In Section~\ref{section.topological.structure} we consider the question of the closure of isolated points and prove that the Cantor-Bendixson rank is infinite. Finally, Section~\ref{open.questions} presents questions that we leave open.\newline

\paragraph*{Acknowledgements} We thank Mathieu Sablik for pointing us in the direction of Pavlov and Schmieding's article, and the fruitful conversations about it. We also thank Nicanor Carrasco for the idea that the Robinson shift is not isolated, on which this article is based. 

\newpage

\section{Background\label{section.background}}

\subsection{{Shifts}} 

\subsubsection{Elementary definitions}

For any finite set $\mathcal{A}\subset\mathbb{Z}$ and integer $d \ge 1$, the set $\mathcal{A}^{\mathbb{Z}^d}$ is called the $d$-dimensional \textit{full shift} on alphabet $\mathcal{A}$. Elements in this set are maps $x:\mathbb{Z}^d\to\mathcal{A}$ that we call $d$-dimensional \textit{configurations}. For such a configuration $x$, the value of $x$ for $\textbf{v}\in\mathbb{Z}^d$ is denoted by $x_\textbf{v}$. For two configurations $x,y$
we set 
\[\delta(x,y) := 2^{ - \inf \{\lVert\textbf{n}\rVert_{\infty}\mid\textbf{n} \in \mathbb{Z}^d : x_{\textbf{n}} \neq y_{\textbf{n}}\}},\]
where $\lVert\textbf{v}\rVert_\infty=\max_{k\leq d}|v_k|$ for every $\textbf{v}$ in $\mathbb{Z}^d$. For a configuration $x$ and a shift $Y$ we denote by $\delta(x,Y)$ the number $\min_{y\in Y}\delta(x,y)$.

We denote by  $\sigma_d$ the $\mathbb{Z}^d$-action on $\mathcal{A}^{\mathbb{Z}^d}$such that for all $\textbf{u},\textbf{v} \in \mathbb{Z}$ and $x$ a d-dimensional configuration, 
$\sigma_d^{\textbf{u}} (x)_{\textbf{v}} = 
x_{\textbf{u}+\textbf{v}}$.

% In the remainder of this text 
% we fix some integer $d \ge 2$
% and 

In the following, we simplify the notation $\sigma_d$ into $\sigma$, as it will be clear from the configuration it applies to, which action it designates.

A $d$-dimensional \textit{shift} is a compact (for the topology induced by the metric $\delta$) 
subset $X$ of $\mathcal{A}^{\mathbb{Z}^d}$ such that for all $\textbf{u} \in \mathbb{Z}^d$,  $\sigma^{\textbf{u}}(X) \subset X$. A \textit{pattern} on a finite set $\mathcal{A}$ is a map $p: \mathbb{U} \rightarrow \mathcal{A}$, of a  domain $\mathbb{U}$ of $\mathbb{Z}^d$. The domain $\mathbb{U}$ is called the \textit{support} of $p$. For every $n\in\N$, we set $\mathbb{B}_n^d:=\llbracket-n,n\rrbracket^d$ and  $\mathbb{U}_n^d:=\llbracket0,n-1\rrbracket^d$. A pattern is called \textit{finite} when its support is finite. In the remainder of the text, we assume that patterns are finite, unless stated otherwise.
We say that a pattern $p$ on support $\mathbb{U}$ \textit{appears} in a configuration $x$ if there is some $\textbf{u} \in \mathbb{Z}^d$ 
such that for all $\textbf{v} \in \mathbb{U}$, $x_{\textbf{v} + \textbf{u}} = p(\textbf{v})$.
% \footnote{\textcolor{red}{\`A v\'erifier la d\'efinition de $p$, car $p(\textbf{v})$ n'est pas n\'ecessairement bien d\'efini.}} 
We denote this by $p\sqsubset x$. We say that $p$ is \textit{globally admissible} for a shift $X$ when it appears in at least one configuration of $X$. For a shift $X$, and $\mathbb{U}$ a subset of $\mathbb{Z}^d$, 
we denote by $\mathscr{L}_{\mathbb{U}}(X)$ the set of patterns on $\mathbb{U}$ which 
are globally admissible 
for $X$. The \textit{language} of $X$ is the set, denoted by $\mathscr{L}(X)$, defined as the union of the sets $\mathscr{L}_{\mathbb{U}}(X)$, where $\mathbb{U}$ runs over all finite subsets of $\mathbb{Z}^d$. We endow the set $\mathscr{L}(X)$, as well as any $\mathcal{L}_{\mathbb{U}}(X)$ where $\mathbb{U}$ is infinite, with the prodiscrete topology.

For any finite set of patterns $\mathcal{F}$ on the same alphabet $\mathcal{A}$, let us denote by $X_{\mathcal{F}}$ the shift on alphabet $\mathcal{A}$ obtained by forbidding the patterns of $\mathcal{F}$: 
\[X_{\mathcal{F}} = \{x \in X : \forall p \in \mathcal{F}, p \not \sqsubseteq x\}.\]
Patterns in $\mathcal{F}$ are called \textit{forbidden patterns} (for $X_\mathcal{F}$), while a pattern that is not in $\mathcal{F}$ is called an \textit{allowed pattern} (for $X_\mathcal{F}$).

We say that a shift $X$ is of \textit{finite type} when there exists a finite set $\mathcal{F}$ such that $X = X_{\mathcal{F}}$. For every configuration $x$, 
we denote by $\mathcal{O}(x)$ the set $\{\sigma^{\textbf{u}}(x) : \textbf{u} \in \mathbb{Z}^d\}$ and call it the \textit{orbit} of $x$.

For two shifts $X,X'$ on alphabets $\mathcal{A},\mathcal{A}'$, we denote by $X \times X'$ the shift on 
alphabet $\mathcal{A} \times \mathcal{A}'$ such that:
\[X \times X' = \left\{z : \exists x \in X, x' \in X', \forall \textbf{u} \in \mathbb{Z}^d, z_{\textbf{u}} = (x_{\textbf{u}},x'_{\textbf{u}})\right\}.\]

A subsystem of $X$ is simply a non-empty closed subset of this shift which is invariant by the shift action. We say that $X$ is \textit{minimal} when its only subsystem is $X$ itself.
We say that it is \textit{transitive} when there exists some $x \in X$ such that $\overline{\mathcal{O}(x)} = X$.

A map $\varphi:X\to Y$, where 
$X,Y$ are two $d$-dimensional shifts, is called a \textit{conjugacy} when it is a homeomorphism such that for all $\boldsymbol{u} \in \mathbb{Z}^d$, $\varphi \circ \sigma_d^{\boldsymbol{u}} = \sigma_d^{\boldsymbol{u}} \circ \varphi$.

% \comm{S}{Do we really need to define conjugacy ? Besides, a conjugacy commutes with the shift.}
% \comm{A}{It is asked by the reviewer (46) :/}

\subsubsection{The space $\mathcal{S}^d$ of $d$-dimensional shifts}

From now on, we denote by $\mathcal{S}^d$ the set of all $d$-dimensional shifts.
For any two shifts, $X,Z$, we set 
\[\delta_H(X,Z) := \max \left(\sup_{x \in X} \inf_{z \in Z} \delta(x,z), \sup_{z \in Z} \inf_{x \in X} \delta(z,x)\right)\]
The function $\delta_H$ is a metric and it is called the \textit{Hausdorff metric}. In the following, we will use the straightforward fact that 
a sequence of shifts $(X_n)_{n \ge 0}$ converges towards $X$ for the metric $\delta_H$ if and only if for all finite subsets $\mathbb{U}$ of $\mathbb{Z}^d$, there exists $m$ such that for all $n \ge m$, $\mathscr{L}_{\mathbb{U}}(X_n) = \mathscr{L}_{\mathbb{U}}(X)$.

\subsection{Constructions\label{section.constructions}}

In this section, we present some elements often used in constructions of higher dimensional shifts, and used as well in the present article. 

\paragraph{Constructions by layers}

A general paradigm for construction of higher dimensional shifts is the one 
of \textit{layers}. By \textit{adding a layer} to a shift $X$ we mean considering a shift $Y$ included in $X \times Z$, where $Z$ is a shift called the added layer. 
When $X$ and $Z$ are shifts of finite type and there is a finite set of patterns $\mathcal{F}$ such that $Y$ is the set of configurations in $X \times Z$ in which no pattern of $\mathcal{F}$ appear, then $Y$ is also of finite type. 
We say that a layer is \textit{accessory} when for all $x \in X$ there is a unique configuration $z\in Z$ such that $(x,z)$ is in $Y$.

\paragraph{Robinson shift}

Another element often used in constructions is a shift due to R. Robinson~\cite{Robinson}. As this construction is quite involved and long to unpack, and  since we use it in constructions which serve as examples, and not in the proofs of core results, we will refer to a past work of the first author~\cite{Gangloff} and to the original article \cite{Robinson} for details about this construction and keep its exposition relatively informal. \bigskip 

We denote this shift by $X_R$. It can be described in terms of layers. Its alphabet is the following set of symbols and their rotations by angles $\frac{\pi}{2}$, $\pi$ or $\frac{3\pi}{2}$:

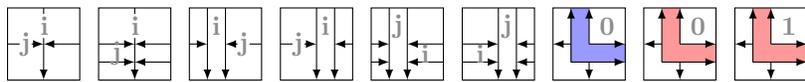
\[\begin{tikzpicture}[scale=0.8]
\draw (0,0) rectangle (1.2,1.2) ;
\draw (0.6,1.2) -- (0.6,1.1);
\draw [-latex] (0.6,0.7) -- (0.6,0) ;
\draw (0,0.6) -- (0.2,0.6);
\draw [-latex] (0.4,0.6) -- (0.6,0.6) ; 
\draw [-latex] (1.2,0.6) -- (0.6,0.6) ;
\node[scale=1,color=gray!90] at (0.6125,0.9) {\textbf{i}};
\node[scale=1,color=gray!90] at (0.3,0.6) {\textbf{j}};
\end{tikzpicture} \ \
\begin{tikzpicture}[scale=0.8]
\draw (0,0) rectangle (1.2,1.2) ;
\draw (0.6,1.2) -- (0.6,1.1);
\draw [-latex] (0.6,0.7) -- (0.6,0) ;
\draw [-latex] (0,0.6) -- (0.6,0.6) ; 
\draw [-latex] (0,0.3) -- (0.6,0.3) ; 
\draw [-latex] (1.2,0.6) -- (0.6,0.6) ;
\draw [-latex] (1.2,0.3) -- (0.6,0.3) ;
\node[scale=1,color=gray!90] at (0.6125,0.9) {\textbf{i}};
\node[scale=1,color=gray!90] at (0.3,0.45) {\textbf{j}};
\end{tikzpicture} \ \ \begin{tikzpicture}[scale=0.8]
\draw (0,0) rectangle (1.2,1.2) ;
\draw [-latex] (0.6,1.2) -- (0.6,0) ;
\draw [-latex] (0.3,1.2) -- (0.3,0) ; 
\draw [-latex] (0,0.6) -- (0.3,0.6) ; 
\draw [-latex] (0.8,0.6) -- (0.6,0.6) ;
\draw (1.2,0.6) -- (1,0.6);
\node[scale=1,color=gray!90] at (0.45,0.9) {\textbf{i}};
\node[scale=1,color=gray!90] at (0.9,0.6) {\textbf{j}};
\end{tikzpicture} \ \  \begin{tikzpicture}[scale=0.8]
\draw (0,0) rectangle (1.2,1.2) ;
\draw [-latex] (0.6,1.2) -- (0.6,0) ;
\draw [-latex] (0.9,1.2) -- (0.9,0) ; 
\draw (0,0.6) -- (0.2,0.6);
\draw [-latex] (0.4,0.6) -- (0.6,0.6) ; 
\draw [-latex] (1.2,0.6) -- (0.9,0.6) ;
\node[scale=1,color=gray!90] at (0.75,0.9) {\textbf{i}};
\node[scale=1,color=gray!90] at (0.3,0.6) {\textbf{j}};
\end{tikzpicture} \ \ 
\begin{tikzpicture}[scale=0.8]
\draw (0,0) rectangle (1.2,1.2) ;
\draw [-latex] (0.6,1.2) -- (0.6,0) ;
\draw [-latex] (0.3,1.2) -- (0.3,0) ; 
\draw [-latex] (0,0.6) -- (0.3,0.6) ; 
\draw [-latex] (1.2,0.6) -- (0.6,0.6) ;
\draw [-latex] (0,0.3) -- (0.3,0.3) ; 
\draw [-latex] (1.2,0.3) -- (0.6,0.3) ;
\node[scale=1,color=gray!90] at (0.9,0.45) {\textbf{i}};
\node[scale=1,color=gray!90] at (0.45,0.9) {\textbf{j}};
\end{tikzpicture} \ \
\begin{tikzpicture}[scale=0.8]
\draw (0,0) rectangle (1.2,1.2) ;
\draw [-latex] (0.6,1.2) -- (0.6,0) ;
\draw [-latex] (0.9,1.2) -- (0.9,0) ; 
\draw [-latex] (0,0.6) -- (0.6,0.6) ; 
\draw [-latex] (1.2,0.6) -- (0.9,0.6) ;
\draw [-latex] (0,0.3) -- (0.6,0.3) ; 
\draw [-latex] (1.2,0.3) -- (0.9,0.3) ;
\node[scale=1,color=gray!90] at (0.3,0.45) {\textbf{i}};
\node[scale=1,color=gray!90] at (0.75,0.9) {\textbf{j}};
\end{tikzpicture} \ \ \begin{tikzpicture}[scale=0.8]
\fill[blue!40] (0.3,0.3) rectangle (0.6,1.2) ;
\fill[blue!40] (0.3,0.3) rectangle (1.2,0.6) ;
\draw (0,0) rectangle (1.2,1.2) ;
\draw [-latex] (0.3,0.3) -- (0.3,1.2) ; 
\draw [-latex] (0.3,0.3) -- (1.2,0.3) ;
\draw [-latex] (0.6,0.6) -- (0.6,1.2) ; 
\draw [-latex] (0.6,0.6) -- (1.2,0.6) ; 
\draw [-latex] (0.3,0.6) -- (0,0.6) ; 
\draw [-latex] (0.6,0.3) -- (0.6,0) ;
\node[scale=1,color=gray!90] at (0.9,0.9) {\textbf{0}};
\node[scale=1] at (0.5,0.5) {\textbf{0}};
\end{tikzpicture} \ \ \begin{tikzpicture}[scale=0.8]
\fill[red!40] (0.3,0.3) rectangle (0.6,1.2) ;
\fill[red!40] (0.3,0.3) rectangle (1.2,0.6) ;
\draw (0,0) rectangle (1.2,1.2) ;
\draw [-latex] (0.3,0.3) -- (0.3,1.2) ; 
\draw [-latex] (0.3,0.3) -- (1.2,0.3) ;
\draw [-latex] (0.6,0.6) -- (0.6,1.2) ; 
\draw [-latex] (0.6,0.6) -- (1.2,0.6) ; 
\draw [-latex] (0.3,0.6) -- (0,0.6) ; 
\draw [-latex] (0.6,0.3) -- (0.6,0) ;
\node[scale=1,color=gray!90] at (0.9,0.9) {\textbf{0}};
\node[scale=1] at (0.5,0.5) {\textbf{1}};
\end{tikzpicture} \ \ \begin{tikzpicture}[scale=0.8]
\fill[red!40] (0.3,0.3) rectangle (0.6,1.2) ;
\fill[red!40] (0.3,0.3) rectangle (1.2,0.6) ;
\draw (0,0) rectangle (1.2,1.2) ;
\draw [-latex] (0.3,0.3) -- (0.3,1.2) ; 
\draw [-latex] (0.3,0.3) -- (1.2,0.3) ;
\draw [-latex] (0.6,0.6) -- (0.6,1.2) ; 
\draw [-latex] (0.6,0.6) -- (1.2,0.6) ; 
\draw [-latex] (0.3,0.6) -- (0,0.6) ; 
\draw [-latex] (0.6,0.3) -- (0.6,0) ;
\node[scale=1,color=gray!90] at (0.9,0.9) {\textbf{1}};
\node[scale=1] at (0.5,0.5) {\textbf{1}};
\end{tikzpicture},\]
where $i$ and $j$ can be $0$ or $1$. The corner symbols with \textbf{0} at the center are referred to as \emph{blue corners} and the ones with \textbf{1} at the center are referred to as \emph{red corners}.
The shift $X_R$ is defined by a set of forbidden patterns which can be derived from the following rules: 
\begin{enumerate} \item Outgoing and incoming arrows must correspond for two adjacent symbols. For instance, the following patterns are respectively forbidden (left) and allowed (right): 
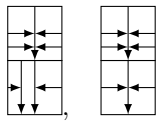
\[\begin{tikzpicture}[scale=0.6]
\draw (0,1.2) rectangle (1.2,2.4) ;
\draw [-latex] (0.6,2.4) -- (0.6,1.2) ;
\draw [-latex] (0,1.8) -- (0.6,1.8) ; 
\draw [-latex] (0,1.5) -- (0.6,1.5) ; 
\draw [-latex] (1.2,1.8) -- (0.6,1.8) ;
\draw [-latex] (1.2,1.5) -- (0.6,1.5) ;
\draw (0,0) rectangle (1.2,1.2) ;
\draw [-latex] (0.6,1.2) -- (0.6,0) ;
\draw [-latex] (0.3,1.2) -- (0.3,0) ; 
\draw [-latex] (0,0.6) -- (0.3,0.6) ; 
\draw [-latex] (1.2,0.6) -- (0.6,0.6) ;
\end{tikzpicture}, \quad \begin{tikzpicture}[scale=0.6]
\draw (0,0) rectangle (1.2,1.2) ;
\draw [-latex] (0.6,1.2) -- (0.6,0) ;
\draw [-latex] (0,0.6) -- (0.6,0.6) ; 
\draw [-latex] (1.2,0.6) -- (0.6,0.6) ;
\draw (0,1.2) rectangle (1.2,2.4) ;
\draw [-latex] (0.6,2.4) -- (0.6,1.2) ;
\draw [-latex] (0,1.8) -- (0.6,1.8) ; 
\draw [-latex] (0,1.5) -- (0.6,1.5) ; 
\draw [-latex] (1.2,1.8) -- (0.6,1.8) ;
\draw [-latex] (1.2,1.5) -- (0.6,1.5) ;
\end{tikzpicture}.\]
\item In every $2 \times 2$ square pattern (meaning a pattern on a translate of $\{0,1\}^2$) there is a blue corner and the presence of a blue corner in a position $\textbf{u} \in \mathbb{Z}^2$ 
forces the presence of a blue corner on positions $\textbf{u}+(0,2), \textbf{u}-(0,2), 
\textbf{u}+(2,0)$ and $\textbf{u}-(2,0)$.
\item The value $i$ is transmitted horizontally, meaning that two consecutive positions $(n,m)$ and $(n+1,m)$ have the same $i$, while the value $j$ is transmitted vertically.
\item On a six arrows symbol or a five arrows symbol, for instance the respective patterns:
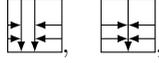
\[\begin{tikzpicture}[scale=0.6]
\draw (0,0) rectangle (1.2,1.2) ;
\draw [-latex] (0.6,1.2) -- (0.6,0) ;
\draw [-latex] (0.3,1.2) -- (0.3,0) ; 
\draw [-latex] (0,0.6) -- (0.3,0.6) ; 
\draw [-latex] (1.2,0.6) -- (0.6,0.6) ;
\draw [-latex] (0,0.3) -- (0.3,0.3) ; 
\draw [-latex] (1.2,0.3) -- (0.6,0.3) ;
\end{tikzpicture},\quad \begin{tikzpicture}[scale=0.6]
\draw (0,0) rectangle (1.2,1.2) ;
\draw [-latex] (0.6,1.2) -- (0.6,0) ;
\draw [-latex] (0,0.6) -- (0.6,0.6) ; 
\draw [-latex] (0,0.3) -- (0.6,0.3) ; 
\draw [-latex] (1.2,0.6) -- (0.6,0.6) ;
\draw [-latex] (1.2,0.3) -- (0.6,0.3) ;
\end{tikzpicture},\]
we have $i \neq j$.
\end{enumerate}

As a consequence, $X_R$ is a shift of finite type.

%\begin{notation}\label{notation.ball}
%For all $n,d$, let us denote by $\mathbb{B}_n^d$
%the set $\llbracket -n , n \rrbracket ^d$, and set $\mathbb{U}^{d}_{n}:=\llbracket 0,n-1 \rrbracket^d$.
%\end{notation}

\begin{figure}[h!]
\[\begin{tikzpicture}[scale=0.4]

\begin{scope}
\robibluebastgauche{0}{0};
\robibluebastgauche{8}{0};
\robibluebastgauche{0}{8};
\robibluebastgauche{8}{8};
\robiredbasgauche{2}{2};
\robiredbasgauche{6}{6};
\robibluehauttgauche{0}{4};
\robibluehauttgauche{8}{4};
\robibluehauttgauche{0}{12};
\robibluehauttgauche{8}{12};
\robiredhautgauche{2}{10};
\robibluebastdroite{4}{0};
\robibluebastdroite{12}{0};
\robibluebastdroite{4}{8};
\robibluebastdroite{12}{8};
\robiredbasdroite{10}{2};
\robibluehauttdroite{4}{4};
\robibluehauttdroite{12}{4};
\robibluehauttdroite{4}{12};
\robibluehauttdroite{12}{12};
\robiredhautdroite{10}{10};
\robitwobas{2}{0}
\robitwobas{10}{0}
\robitwobas{6}{2}
\robitwogauche{0}{2}
\robitwogauche{2}{6}
\robitwogauche{0}{10}
\robitwodroite{12}{2}
\robitwodroite{12}{10}
\robitwohaut{2}{12}
\robitwohaut{10}{12}
\robionehaut{6}{0}
\robionehaut{6}{4}
\robionegauche{0}{6}
\robionegauche{4}{6}
\robisixbas{2}{8}
\robisixdroite{4}{2}
\robisixhaut{2}{4}
\robisevendroite{4}{10}
\robithreehaut{6}{8}
\robisixhaut{6}{10}
\robithreehaut{6}{12}
\robisixgauche{8}{2}
\robisevengauche{8}{10}
\robisevenbas{10}{8}
\robisevenhaut{10}{4}
\robisixdroite{10}{6}
\robithreedroite{8}{6}
\robithreedroite{12}{6}
\draw (0,0) rectangle (14,14);
\foreach \x in {1,...,6} \draw (0,2*\x) -- (14,2*\x);
\foreach \x in {1,...,6} \draw (2*\x,0) -- (2*\x,14);
\node[scale=0.9] at (0.825,0.825) {\textbf{0}};
\node[scale=0.9] at (8.825,0.825) {\textbf{0}};
\node[scale=0.9] at (8.825,8.825) {\textbf{0}};
\node[scale=0.9] at (0.825,8.825) {\textbf{0}};
\node[scale=0.9] at (5.175,5.175) {\textbf{0}};
\node[scale=0.9] at (13.175,13.175) {\textbf{0}};
\node[scale=0.9] at (13.175,5.175) {\textbf{0}};
\node[scale=0.9] at (5.175,13.175) {\textbf{0}};

\node[scale=0.9] at (13.175,0.825) {\textbf{0}};
\node[scale=0.9] at (13.175,8.825) {\textbf{0}};
\node[scale=0.9] at (5.175,8.825) {\textbf{0}};
\node[scale=0.9] at (5.175,0.825) {\textbf{0}};

\node[scale=0.9] at (0.825,13.175) {\textbf{0}};
\node[scale=0.9] at (8.825,13.175) {\textbf{0}};
\node[scale=0.9] at (8.825,5.175) {\textbf{0}};
\node[scale=0.9] at (0.825,5.175) {\textbf{0}};

\node[scale=0.9] at (2.825,2.825) {\textbf{1}};
\node[scale=0.9] at (2.825,11.175) {\textbf{1}};
\node[scale=0.9] at (11.175,2.825) {\textbf{1}};
\node[scale=0.9] at (6.825,6.825) {\textbf{1}};
\node[scale=0.9] at (11.175,11.175) {\textbf{1}};
\end{scope}
\end{tikzpicture}\]
\caption{The south west supertile of order two (some 0s and 1s are omitted for clarity).}\label{figure.order2supertile}
\end{figure}
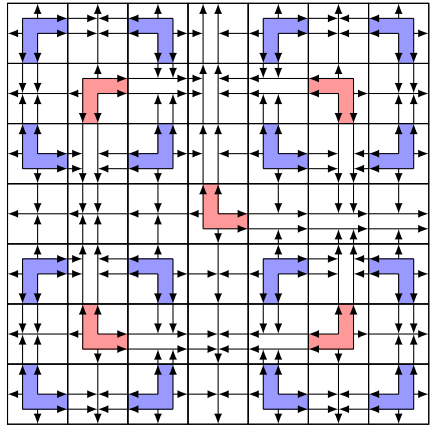 \bigskip

\textit{Finite supertiles ---} In any configuration of $X_R$, the blue corners cause the presence of some particular square patterns that are called \textit{supertiles}, which themselves cause larger supertiles that are similar to the first ones, etc. Let us define these supertiles formally, by induction. For all $n\in\mathbb{N}$, we denote by $St_{sw} (n)$, $St_{sw} (n)$, $St_{sw}(n)$ and $St_{sw}(n)$ 
the \textit{south west, south east, north west, north east supertiles of order $n$}, respectively. For $n=0$ these supertiles are defined as follows:
\[
St_{sw} (0)=\vbox to 14pt{\hbox{
\begin{tikzpicture}[scale=0.35]
 \robibluebastgauche{0}{0};
\end{tikzpicture}}},\quad
St_{se} (0)=\vbox to 14pt{\hbox{
\begin{tikzpicture}[scale=0.35]
 \robibluebastdroite{0}{0};
\end{tikzpicture}}},\quad
St_{nw} (0)=\vbox to 14pt{\hbox{
\begin{tikzpicture}[scale=0.35]
 \robibluehauttgauche{0}{0};
\end{tikzpicture}}},\quad
St_{ne} (0)=\vbox to 14pt{\hbox{
\begin{tikzpicture}[scale=0.35]
 \robibluehauttdroite{0}{0};
\end{tikzpicture}}}.
\] 

For all $n\in\mathbb{N}$, the support of the supertile $St_{sw}(n+1)$ 
(resp. $St_{se} (n+1)$, $St_{nw} (n+1)$, $St_{ne} (n+1)$) is set as
$\mathbb{U}^{(2)}_{2^{n+2}-1}$. These supertiles are constructed as follows: 

\begin{enumerate} 
\item For $\textbf{u}=(2^{n+1}-1,2^{n+1}-1)$, the 
symbols $St_{sw} (n+1) _{\textbf{u}}$, $St_{se} (n+1) _{\textbf{u}}$, $St_{nw} (n+1)_{\textbf{u}}$,
$St_{ne} (n+1) _{\textbf{u}}$ are
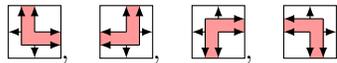
\[
\begin{tikzpicture}[scale=0.35]
 \robiredbasgauche{0}{0};
\end{tikzpicture},\quad
\begin{tikzpicture}[scale=0.35]
 \robiredbasdroite{0}{0};
\end{tikzpicture},\quad
\begin{tikzpicture}[scale=0.35]
 \robiredhautgauche{0}{0};
\end{tikzpicture},\quad
\begin{tikzpicture}[scale=0.35]
 \robiredhautdroite{0}{0};
\end{tikzpicture},
\]
respectively.
\item The restriction
to $\mathbb{U}^{(2)}_{2^{n+1}-1}$ (resp. 
$(2^{n+1},0)+\mathbb{U}^{(2)}_{2^{n+1}-1}$, 
 $(0,2^{n+1})+\mathbb{U}^{(2)}_{2^{n+1}-1}$,
$(2^{n+1},2^{n+1})+\mathbb{U}^{(2)}_{2^{n+1}-1}$)
is 
$St_{sw}(n)$ (resp. $St_{se} (n)$, $St_{nw} (n)$,
$St_{ne} (n)$).
\item Each of the other positions 
(forming a cross whose center in at the center of the square) is filled with one symbol among the following, up to rotation (they are uniquely determined by their neighboring symbols): 
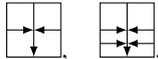
\[\begin{tikzpicture}[scale=0.6]
\draw (0,0) rectangle (1.2,1.2) ;
\draw [-latex] (0.6,1.2) -- (0.6,0) ;
\draw [-latex] (0,0.6) -- (0.6,0.6) ; 
\draw [-latex] (1.2,0.6) -- (0.6,0.6) ;
\end{tikzpicture}, \quad \begin{tikzpicture}[scale=0.6]
\draw (0,0) rectangle (1.2,1.2) ;
\draw [-latex] (0.6,1.2) -- (0.6,0) ;
\draw [-latex] (0,0.6) -- (0.6,0.6) ; 
\draw [-latex] (0,0.3) -- (0.6,0.3) ; 
\draw [-latex] (1.2,0.6) -- (0.6,0.6) ;
\draw [-latex] (1.2,0.3) -- (0.6,0.3) ;
\end{tikzpicture}.\] 
\end{enumerate} 

An example of a supertile is displayed on Figure~\ref{figure.order2supertile}. \bigskip 

\textit{Infinite supertiles ---}
Let $x$ be a configuration in $X_R$ and consider 
the equivalence relation $\sim_x$ on $\mathbb{Z}^2$ defined by:
$\textbf{i} \sim_x \textbf{j}$ $\Leftrightarrow$ there is a finite
supertile in $x$ on a support which contains $\textbf{i}$ and $\textbf{j}$.
An \textit{infinite order} supertile is an infinite pattern whose support is an equivalence 
class of this relation. Each configuration satisfies one of the following:
\begin{itemize}
\item[(i)] It has a unique infinite order supertile on support $\mathbb{Z}^2$.
\item[(ii)] It has two infinite order supertiles separated by an infinite row $\mathbb{Z} \times \{m\}$ or a column  $\{m\} \times \mathbb{Z}$ with only three-arrows symbols or only four arrows symbols on it.  
% \textcolor{cyan}{Different to what happens in configurations satisfying (i) or (iii), order $n$ finite supertiles in one infinite supertile might be nonaligned with the same-order-supertiles present in the other infinite supertile.}
In such a configuration, there might be some finite $n$ such that order $n$ supertiles in one of the infinite supertiles might not be aligned with the order $n$ supertiles in the other (whereas %this is the case 
in a configuration which satisfies (i) or (iii), order $n$ supertiles are aligned with respect to all infinite supertiles).
% \textcolor{cyan}{If one of the infinite supertiles is shifted with respect to the other we obtain another (different) configuration satisfying the same property. }
% \textcolor{gray}{\footnote{\textcolor{red}{est-ce qu'on a vraiment besoin de \c{c}a ?}}}
\item[(iii)] It has four infinite order supertiles, separated by an infinite cross, that is $\mathbb{C}_{m,n} := \mathbb{Z} \times \{m\} \bigcup \{n\} \times \mathbb{Z}$ for some $m,n \in \mathbb{Z}$, such that $(m,n)$ is superimposed with a red corner, 
a six-arrows symbol or a five-arrows symbol. In any case the symbols on $\mathbb{C}_{m,n} \backslash \{(m,n)\}$ are uniquely determined by their surrounding symbols.
\end{itemize} 

% \comm{S}{26,27}

\textit{Minimal version ---}
Note that $X_R$ is not a minimal shift. The main reason is that provided a configuration of type (ii), any configuration obtained from it by shifting one of its infinite supertile relatively to the other is also in  $X_R$. There are patterns across the infinite line separating the infinite supertile which do not appear in every configuration of $X_R$. However, it is possible to add a layer to $X_R$ 
so that in configurations which satisfy the condition (ii), the supertiles of the same (finite) order across the infinite row or column are aligned and the obtained shift is a minimal shift of finite type. Roughly speaking, the idea is to use the accessory layer to send ``signals'' between supertiles forcing them to be aligned. This is a %known fact 
standard technique and a detailed construction can be found in \cite{Gangloff}, for instance. We denote this shift by $X_R^{+}$. \bigskip 

% \comm{A}{Peut-\^ere qu'on peut faire mieux ici pour addresser plus facilement apr\`es le commentaire (43) (exemple inf ssyst)} 

% \comm{S}{Raj justification pour minimal sft rob et ref papier min}

\textit{Additional terminology ---}
Provided a configuration of $X_R$, for each supertile of order $n$, the position in $\mathbb{Z}^2$ of the central symbol is called a \textit{site} of order $n$. We call \textit{cell} of order $n$
the subset of $\mathbb{Z}^2$
enclosed by sites of order $2n+1$ at the center of which there is another site. We call \textit{cytoplasm} of a cell of order $n$ the set of positions in this cell which do not belong to any cell of smaller order. The \textit{border} of a cell is the set of its elements which do not have four neighbors in it. A \textit{free column} (resp. row) of a cell $\mathbb{C}$ is a subset of the form $(\mathbb{Z} \times \{m\})\cap \mathbb{C}$ (resp. $(\{m\}\times\mathbb{Z})\cap \mathbb{C}$) which does not intersect a cell of smaller order. \bigskip 

\textit{Subsystems of $X_R$ ---}
We denote by $X_{R}^{\infty}$ the set of configurations which satisfy $(i)$ or $(iii)$. 
It is a minimal shift. Furthermore, for all $n$ we denote by $X_R^n$ the subsystem of $X_R$ obtained by enforcing, in configurations satisfying $(ii)$, the order $k \le n$
supertiles across the separating row or column, 
to be aligned. It is straightforward that $X_R^n$ is of finite type for all $n$, that the sequence $(X_R^n)_n$ is strictly decreasing, and that
\[\bigcap_n X_R^n = X_R^{\infty}.\]

For the remainder of the text we denote by $\textbf{e}_1, \ldots, \textbf{e}_d$ the canonical basis of $\mathbb{Z}^d$.

\paragraph{Encoding finite alphabets into $\mathbb{Z}$.} In the text, we provide examples of shifts on some alphabets $\mathcal{A} \not \subset \mathbb{Z}$. For each one of these, we implicitly re-encode the shift so that the alphabet is included in $\mathbb{Z}$. By re-encoding we mean the following: consider $X$ a shift on alphabet $\mathcal{A}$, and $\phi : \mathcal{A} \rightarrow \mathcal{A}'$ a bijection. The re-encoded shift is the image of $X$ by the map $x \mapsto (\phi(x_{\boldsymbol{u}}))_{\boldsymbol{u} \in \mathbb{Z}^d}$.

\section{Characterization of isolated points in $\mathcal{S}^d$\label{section.isolated.shifts}}

In order to investigate generic shifts in higher dimension, the approach we have taken is to 
first understand isolated points in $\mathcal{S}^d$. We characterize these points in the present section, and this characterization will help us prove that the closure of the set of isolated shifts is not $\mathcal{S}^d$. The characterization is as in \cref{thm:A}: a shift is isolated if and only if it is of finite type, it has finitely many maximal subsystems and each of its subsystems is included in a maximal subsystem. We prove it in Section~\ref{section.characterization.isolated}, and as an application we recover the characterization in one dimension formulated by R. Pavlov and S. Schmieding. 
Using a construction based on the Robinson shift, we construct a shift with finitely many maximal subsystems, but infinitely many subsystems (Section~\ref{section.example.inf.subsystems}). We remark that such a  system does not exist in the space of $\mathbb{Z}$-shifts.

To the best of our knowledge, the notion of maximal subsystem
involved naturally in the characterization of isolated shifts has not been studied elsewhere. We devote Section~\ref{section.study.maximal} to the study of this notion in and of itself. Section~\ref{section.auxiliary} contains auxiliary results on topology which are useful in the present section.

\subsection{{Auxiliary results}\label{section.auxiliary}}

Results presented throughout this subsection are elementary, but proofs are included for completeness.
%For a configuration $x$ and a shift $X$, we denote by $\delta(x,X)$ the number $\inf_{z \in X} \delta(x,z)$. For $\epsilon > 0$, we denote by $B_{\epsilon}(x,X)$ the set of configurations $z \in X$ such that $\delta(x,z) < \epsilon$.

\begin{lemma}\label{lemma.intersection}
Let $(X_n)_{n\ge0}$ be a sequence of shifts in $\mathcal{S}^d$ such that for all $n \ge 0$, $X_{n+1} \subset X_n$ and set $X_{\infty}:= \bigcap_n X_n$. Then $X_n$ converges to $X_{\infty}$ for the Hausdorff topology.
\end{lemma}

\begin{proof}
Let us assume, on the contrary, that there exists $\epsilon>0$ and infinitely many integers $n$ such that 
there exists $x_n \in X_n$ with $\delta(x_n,X_{\infty}) \ge \epsilon$. %Because $(X_n)_{n \ge 0}$ is non-increasing, all but finitely many integers satisfy this. 
By compactness, we can assume that $x_n$ converges to a point $x \in X$. Furthermore, since $(X_n)_{n \ge 0}$ is non-increasing, $x \in X_n$ for all $n$, and thus $x \in X_{\infty}$. On the other hand, because the distance $\delta$ is a continuous function, $\delta(x,X_{\infty}) \ge \epsilon$, which is a contradiction.
\end{proof}

The following is straightforward:

\begin{lemma}\label{lemma.inclusion}
    Consider three shifts $X' \subset X \subset Z$. We have $\delta_H(X',Z) \ge \delta_H(X,Z)$.
\end{lemma}

For $\epsilon > 0$, we denote by $B_{\epsilon}(x,X)$ the set of configurations $z \in X$ such that $\delta(x,z) < \epsilon$. We say that a sequence of shifts $(X_n)_{n\geq0}$ is non-decreasing (resp. non-increasing) if for all $n$ it holds that $X_n\subseteq X_{n+1}$ (resp. $X_{n+1}\subseteq X_n$).

\begin{lemma}\label{lemma.union}
    Consider a shift $X_{\infty}$ and a non-decreasing sequence of shifts $(X_n)_{n \ge 0}$ such that $X_{\infty} = \bigcup_n X_n$. 
    Then $X_n$ converges to $X_{\infty}$.
\end{lemma}

\begin{proof}
    If it were not the case, there would be some $\epsilon>0$ and a subsequence $(n_k)_{k\in\N}$ such that for all $k$ we have $\delta_H(X_{n_k},X_\infty) \ge \epsilon$, and by \cref{lemma.inclusion}, we have that for all $n$, $\delta_H(X_n,X_{\infty}) \ge \epsilon$. For all $n$, since $X_n \subset X_{\infty}$, there exists an element $x_n$ of $X_{\infty}$ such that $B_{\epsilon}(x_n,X_{\infty})$ is contained in $X_{\infty} \backslash X_n$. By compactness, we can assume that $x_n \rightarrow x \in X_{\infty}$. Therefore $B_{\epsilon/2}(x,X_{\infty})$
    does not intersect any $X_n$, which contradicts the fact that $X_{\infty}$ is the union of the sets $X_n$.
\end{proof}

\begin{lemma}\label{lemma.stability}
Let $X$ be a shift and let $U \subset X$ be such that $\sigma^\textbf{u}(U) \subset U$ for all $\textbf{u} \in \mathbb{Z}^d$. The set $\overline{U}$ is a shift.
\end{lemma}

\begin{proof}
    Let $x \in \overline{U}$. There exists a sequence of elements $x_n$ of $U$ which converges to $x$. Because $U$ is invariant under the action of the shift, for all $\textbf{v}\in\mathbb{Z}^d$, $\sigma^{\textbf{v}}(x_n) \in U$. By continuity of $\sigma^{\textbf{v}}$, for all $\textbf{v}\in\mathbb{Z}^d$, $\sigma^{\textbf{v}}(x_n)$ converges to $\sigma^{\textbf{v}}(x)$, which is thus in $\overline{U}$.
\end{proof}

\subsection{Characterization of isolated points in $\mathcal{S}^d$}\label{section.characterization.isolated}

This section is decomposed into three subsections: first we prove that isolated shifts are all of finite type (Section~\ref{section.isolated.sft}); second we exhibit a relation with maximal subsystems (Section~\ref{section.maximal.subsystems}); then we formulate the characterization of isolated shifts in $\mathcal{S}^d$ in Section~\ref{subsection.characterization}.

\subsubsection{{Isolated shifts are of finite type}\label{section.isolated.sft}}

First, we show that any isolated shift is a shift of finite type, and that for any shift of finite type $X$, there is a neighborhood of $X$ which contains only 
subsystems of $X$.

\begin{lemma}\label{lemma.isolated.are.sft}
For all $d \ge 1$, isolated shifts in $\mathcal{S}^d$ are of finite type. 
\end{lemma}

\begin{proof}
 Let $X$ be an isolated shift and assume it is not of finite type. Let $\mathcal{F}$ be a set of forbidden patterns such that $X = X_{\mathcal{F}}$. For all $n \ge 1$, we consider the shift $X_n := X_{\mathcal{F}_n}$, where
$\mathcal{F}_n$ is the set of patterns in $\mathcal{F}$ which appear in a pattern on 
$\mathbb{B}_n^d=\llbracket-n,n\rrbracket^d$. By definition, the sequence $(X_n)_{n \ge 0}$ is non-increasing and takes infinitely many values, otherwise the sequence would be ultimately constant and $X$ would be equal to some $X_n$ and thus be of finite type. We only have left to prove that $X_n$ converges to $X$, and this is a direct consequence of Lemma~\ref{lemma.intersection}.
\end{proof}

\noindent Reciprocally, a shift of finite type is not necessarily isolated: 

\begin{example}
    For any alphabet $\mathcal{A}$ such that $|\mathcal{A}| \ge 2$, the full shift on $\mathcal{A}$ is not isolated in $\mathcal{S}^d$. Indeed, for all $n \ge 0$, set  $X_n := X_{\{p_n\}}$, where $p_n : \mathbb{B}_n^d \rightarrow \mathcal{A}$
    is constant with value $a$. The sequence $(X_n)_{n\ge0}$ converges towards the full shift on alphabet $\mathcal{A}$.
\end{example}

However we have information about 
the systems contained 
in a close enough neighborhood of a shift of finite type.

\begin{lemma}\label{lemma.subsystem}
    A shift $X \in \mathcal{S}^d$ is of finite type if and only if there exists $\epsilon>0$
    such that for all $Z \in \mathcal{S}^d$, if $\delta_H(Z,X) < \epsilon$ then $Z \subset X$.
\end{lemma}

\begin{proof}
    ($\Rightarrow$) There is a finite set of patterns $\mathcal{F}$ such that $X = X_{\mathcal{F}}$. There exists $n$ such that the supports of these patterns are all contained in some translate of $\mathbb{B}_n^d$. Therefore, the ball around $X$ of diameter $2^{-n}$ for the metric $\delta_H$ contains only shifts whose languages have empty intersection with $\mathcal{F}$, meaning that they are subsystems of $X$. ($\Leftarrow$) If $X$ is not of finite type, then for all $\epsilon$ there is some $Z \in \mathcal{S}^d$ such that $\delta_H(Z,X) < \epsilon$ and $Z$ is not contained in $X$. In order to see this it is sufficient 
    to consider the sequence $(X_n)$ such that for all $n$, $X_n$ is defined by forbidding the patterns which are on some support $\mathbb{U}$ amongst the first $n$ subsets of an enumeration of all finite subsets of $\mathbb{Z}^d$ and that are not in $\mathscr{L}_{\mathbb{U}}(X)$.
\end{proof}

%\begin{remark}
%In the previous lemma, %the converse is also %true. This yields %another characterization %of SFT, which might be %generalized to other %contexts.
%\end{remark}

%\begin{question}
%    Can the previous lemma be generalized to topological systems ? (generalizing the notion of SFT using elements of an open base instead of forbidden words)
%\end{question}

\begin{lemma}\label{lemma.subsystem.2}
    Consider a shift of finite type $X$. Every subsystem of $X$ has a neighborhood whose elements are all subsystems of $X$.
\end{lemma}

\begin{proof}
        The proof is similar to the direction $(\Rightarrow)$ in the proof of Lemma \ref{lemma.subsystem}.
\end{proof}

\subsubsection{{Maximal subsystems}\label{section.maximal.subsystems}}

After the observations made in the last section, it is natural to search for a characterization of isolated shifts in terms of their subsystems. As a matter of fact, we have the following:

\begin{proposition}\label{proposition.finite.subsystems}
If a shift of finite type has finitely many subsystems, then it is isolated.
\end{proposition}

\begin{proof}
    Let us consider a shift of finite type $X$ which has finitely many subsystems. As a consequence of Lemma~\ref{lemma.subsystem}, there is a neighborhood of $X$ which contains only subsystems of $X$. Since by hypothesis there are finitely many of them, there is a neighborhood of $X$ which only contains $X$, meaning that it is isolated.
\end{proof}

In the same vein, we have the following lemma which will be useful below:

\begin{lemma}\label{lemma:subsystem.finite.subsystems}
    Every subsystem of a shift of finite type $X$ having a finite number of subsystems is isolated. 
\end{lemma}

\begin{proof}
    As a consequence of Lemma \ref{lemma.subsystem.2}, every subsystem of $X$ has a neighborhood which consists only of subsystems of $X$, meaning that it is finite. Since a shift with finite neighborhood is isolated, this yields the statement.
\end{proof}

\begin{remark}
    Lemma \ref{lemma:subsystem.finite.subsystems} implies that if a shift of finite type has finitely many subsystems, all its subsystems are of finite type.
\end{remark}

One may observe that the proof of Proposition~\ref{proposition.finite.subsystems} can be generalized to 
shifts of finite type which has a finite collection $\mathcal{C}$ of strict subsystems with the property that every strict subsystem is contained in an element of $\mathcal{C}$. This suggest the following definition:

\begin{definition}\label{definition.maxsub.outcast}
    Let $X$ be a shift. We say that a 
    % \textcolor{cyan}{(non-empty)}
    % \footnote{\textcolor{red}{Le reviewer a demand\'e de clarifier si $Z=\emptyset$ peut \^etr\'e un sous syst\`eme maximal ou pas. J'ai mis \c{c}a car sinon il faut changer pas mal de choses apr\`es}} 
    subsystem $Z$ of $X$ is \textbf{maximal} when it is different from $X$ and if $Z' \supset Z$ is another subsystem different from $Z$, then $Z' = X$. We say that a subsystem of $X$ is an \textbf{outcast} if it is not contained in a maximal subsystem.
\end{definition}

Then we can formulate the following: 

\begin{proposition}\label{proposition.finite.maximal}
    If a shift of finite type has finitely many maximal subsystems and no outcast, then it is isolated.
\end{proposition}

\begin{proof}
    Consider a shift $X$ as in the statement of the proposition. Any shift $Z$ close enough to $X$ for the Hausdorff distance is a subsystem of $X$ and thus, included in a maximal subsystem or equal to $X$. All subsystems different from $X$ are at distance larger than the minimum of $ \delta_H(Z',X)$, where $Z'$ is maximal (this comes from Lemma~\ref{lemma.inclusion}) which is positive. Hence $X$ is isolated.
\end{proof}

\begin{remark}
    The hypotheses in Proposition~\ref{proposition.finite.maximal} are satisfied by shifts of finite type which have finitely many subsystems. Proposition~\ref{proposition.finite.subsystems} is thus a consequence of Proposition~\ref{proposition.finite.maximal}.
\end{remark}

We also use the following notation.

\begin{notation}
Let $X$ be a shift. We denote by $\mathcal{M}(X)$ the set of maximal subsystems of $X$, and by $\mathcal{S}(X)$ the set of its strict subsystems.
\end{notation}

\begin{lemma}\label{lemma.two.maximal}
    Let $Z_1,Z_2$ be two different maximal subsystems of a shift $X$. We have $Z_1\cup Z_2=X$.
\end{lemma}

\begin{proof}
    Suppose it is not the case. Then since $Z_1$ and $Z_2$ are different, $Z_1$ is strictly included in $Z_1\cup Z_2$ which is also different from $X$ by hypothesis. This implies that $Z_1$ is not a maximal subsystem of $X$.
\end{proof}

%\textcolor{red}{It follows from \cref{definition.maxsub.outcast} that if $Z_1,Z_2$ are maximal subsystems of $X$, then $Z_1\cup Z_2=X$, otherwise, $Z_1\subset(Z_1\cup Z_2)\neq X$.}

\begin{proposition}\label{proposition.inf.max.sub}
    If a shift of finite type has infinitely many maximal subsystems, 
    then it is not isolated. 
\end{proposition}

\begin{proof}
Let $X$ be a shift.
Let $(Z_n)_{n\ge0}$ be a sequence of maximal subsystems, set $K := \bigcap_n Z_n$, and for all $n$, set 
\[K_n := \bigcup_{k=0}^n \overline{\left(X \backslash Z_k \right)}\]
It is straightforward that $K$ is a shift. This is also the case for $K_n$ for all $n$. Indeed, for all $k$, $X \backslash Z_k$ is invariant by the shift because $Z_k$ is.
As a consequence of Lemma~\ref{lemma.stability}, $\overline{\left(X \backslash Z_k \right)}$ is a shift, thus $K_n$ is a shift.
Furthermore $(K_n)$ is a non-decreasing sequence whose union contains $X \backslash K$. Indeed, this union contains 
the union of the sets $X \backslash Z_k$, which is equal to 
$X \backslash K$.
Therefore 
the sets $K_n$, together with $K$, cover $X$. By this, for all $n$ there exists $m$ such that for all $l \ge m$, 
the Hausdorff distance between $K \bigcup K_l$ and $X$ is smaller than $2^{-n}$ (this comes from Lemma~\ref{lemma.union}). We only have left to prove that $(K_n)$ is not ultimately constant. 
By Lemma~\ref{lemma.two.maximal}, for all $n \neq m$, $Z_n \bigcup Z_m = X$.
This means that $Z_m \supset X \backslash Z_n$ and thus 
$Z_m \supset \overline{X \backslash Z_n}$. This implies that for all $n$, $Z_{n+1}$ contains $K_n$ and thus $K \bigcup K_n$. If $(K_n)$ was ultimately constant, this would contradict the fact that $Z_{n+1}$ is maximal and thus not equal to $X$.

% \textcolor{cyan}{... together with $K$ cover $X$. Observe that for any $n\in\N$ $K_n\subseteq K_{n+1}$, thus $K\cup K_n\subseteq K\cup K_{n+1}$. Denote by $n_\alpha$ the number such that $\delta_H(K\cup K_\alpha,X)=2^{-n_\alpha}$, and observe that $n_\alpha\leq n_{\alpha+1}$. Set $K_0=\emptyset$, $i_0=0$, and for $j>0$ set $i_j=n_{i_{j-1}}+|{\mathcal{A}^{\mathbb{B}^d_{n_{i_{j-1}}}}}|$. We have that $i_0<i_1<\ldots<n_j<\ldots$, and therefore $n_{i_0}<n_{i_1}<\ldots<n_{i_j}<\ldots$. It follows that for all $n\in\N$ there exists $j\in\N$ such that $n<n_{i_j}$ and thus $\delta_H(K\cup K_{i_j},X)=2^{-n_{i_j}}<2^{-n}$. We only have left to prove...}
\end{proof}

% \comm{S}{point (36)}

The following is straightforward.

\begin{lemma}\label{lemma.stability.outcast}
    Consider $Z$ an outcast of a shift $X$. Every $Z' \in \mathcal{S}(X)$ containing $Z$ is also an outcast.
\end{lemma}

\begin{lemma}\label{lemma.distance.outcast}
    Consider $Z$ an outcast of a shift $X$. For every $x \in X \backslash Z$ and every $\epsilon>0$ there exists an outcast $Z'$ of $X$ such that $\delta(x,Z') < \epsilon$ and $Z \subset Z'$.
\end{lemma}

\begin{proof}
Fix $x \in X \backslash Z$. Assume that there exists $\epsilon > 0$ such that for every $Z' \in \mathcal{S}(X)$ with $Z \subset Z'$, $\delta(x,Z') \ge \epsilon$. This implies that the distance between $x$ and the set
\[ \overline{\bigcup_{\underset{Z'' \supset Z}{Z'' \in \mathcal{S}(X)}} Z''}\]
is at least $\epsilon$.
As a consequence, this set is different from $X$ and this means that it is maximal, which is not possible. This implies that there exists $Z' \in \mathcal{S}(X)$ such that $Z \subset Z'$ and $\delta(x,Z') < \epsilon$. Then, by \cref{lemma.stability.outcast}, $Z'$ is also an outcast of $X$.
\end{proof}

\begin{proposition}
    If $Z$ is an outcast of a shift $X$, then there exists a strictly increasing sequence $(Z_n)_{n\in\mathbb{N}}$ of shifts such that $Z_0 = Z$ and $\overline{\cup_{n\ge0}Z_n}=X$.
\end{proposition}

\begin{proof}
    Set $Z_0=Z$ and suppose that $Z_0, \ldots , Z_{n-1}$ have been already constructed for $n \ge 1$ and $\delta_H(Z_{n-1},X) < 2^{-n-1}$.
     There exists a finite set $S_n \subset X$ such that 
    every point of $X$ is at distance strictly less than $2^{-n-1}$ 
    to a point in $S_n$.
    Using Lemma~\ref{lemma.distance.outcast} multiple times, we obtain an outcast $Z_n$ of $X$ containing $Z_{n-1}$ such that for all $x \in S_n$, $\delta(x,Z_n) < 2^{-n-1}$. As a consequence the Hausdorff distance 
    between $X$ and $Z_n$ is less than $2^{-n}$. Since $Z_n$ is an outcast, it is different from $X$. By construction, the sequence $(Z_n)_n$ is non-decreasing and $\overline{\cup_{n\ge0}Z_n}=X$. Since for every $n$, $Z_n$ is different from $X$, it is possible to extract from $(Z_n)$ another sequence $(Z'_n)_n$ such that $Z'_0 = Z$ which is strictly increasing and such that $\overline{\cup_{n\ge0}Z'_n}=X$.
    % \textcolor{cyan}{The sequence $(Z_n)_{n\in\N}$ that arises in this manner is such that for every $n\in\N$ it holds that $Z_n$ is an outcast of $X$, and $Z_n\subset Z_{n-1}$. If the shift $Y=\overline{\cup_{n\geq0}Z_n}$ is different from $X$, then $Y$ is a maximal subsystem of $X$ that contains $Z$. This contradiction completes the proof.}
\end{proof}

% \comm{S}{(37)}

A straightforward consequence of this is the following proposition:

\begin{proposition}\label{proposition.isolated.strict.contained}
    If a shift has an outcast, then it is not isolated.
\end{proposition}

\subsubsection{{Characterization of isolated shifts}\label{subsection.characterization}}
The results in the previous section suggest the following definition.
\begin{definition}\label{definition.star}
    We say that a shift satisfies the condition $(\star)$ when: it has finitely many maximal subsystems and no outcasts.
\end{definition}

\cref{thm:A} is a reformulation of the following.

\begin{theorem}\label{thm.main}
A shift is isolated if and only it satisfies the condition $(\star)$ and is of finite type.
\end{theorem}

\begin{proof}
    This is a consequence of 
\cref{lemma.isolated.are.sft,proposition.finite.maximal,proposition.isolated.strict.contained,proposition.inf.max.sub}.
\end{proof}

From this we can recover the characterization of isolated shifts in $\mathcal{S}^1$. First, we introduce some notation and definitions.

\begin{notation}
    For a one-dimensional shift $X$ and $n\in \mathbb{N}$, we denote by  $G_{X,n}$ the directed graph with vertex set $\mathscr{L}_{\llbracket 0,n-1\rrbracket}(X)$, and where there is an edge from 
    $w_0 \ldots w_{n-1}$
to $w_1 \ldots w_{n}$ whenever $w_0 \ldots w_n \in \mathscr{L}_{\llbracket 0,n\rrbracket}(X)$.
\end{notation}

\begin{definition}
    A directed graph has the \textbf{no-middle-cycle property} if it
contains no \textit{middle cycle}, meaning a cycle $c$ which has both an incoming edge (that connects a vertex not in $c$ to a vertex in $c$) and an outgoing edge (that connects a vertex in $c$ to a vertex not in $c$).
\end{definition}

\begin{remark}
    Observe that if $G_{X,n}$ has the no-middle-cycle property, then for all $m\ge n$, $G_{X,m}$ has the no-middle-cycle property as well.
\end{remark}

\begin{definition}
    A one-dimensional shift X is said to have the no-middle-cycle property when there
exists an integer $n$ for which $G_{X,n}$ has the no-middle-cycle property.
\end{definition}

    For any finite directed graph $G$, we denote by $X_G$ the one-dimensional shift whose alphabet is the vertex set of $G$ and is defined by forbidden patterns $ab$, where there is no edge in $G$ from $a$ to $b$. Following R.Pavlov and S.Schmieding, we call \textit{barbell} a finite directed graph composed of exactly two cycles and a directed path which departs from a vertex of one cycle and ends at a vertex of the other one.

    \begin{lemma}\label{lemma.nmc}
        For any finite directed graph $G$, $X_G$ is isolated in $\mathcal{S}^1$ if and only if $G$ has the no-middle-cycle property. In this case, it has finitely many subsystems.
    \end{lemma}

    \begin{proof}
        $(\Leftarrow)$ If $G$ satisfies the no-middle-cycle property, then every configuration is in some $X_H$, where $H$ is a cycle or a barbell in $G$, and every subsystem 
        is a finite union of such shifts. Thus $X_G$ has finitely many subsystems. Since it is also of finite type, it is isolated in $\mathcal{S}_1$.
    $(\Rightarrow)$ Let us assume that $G$ does not satisfy the no-middle-cycle property.
    Denote by $X_n$ the subsystem of $X_G$ obtained by forbidding a configuration that contains
    the concatenation of $n$ times the same middle cycle. The shifts $X_n$, $n \ge 0$, are all distinct and converge towards $X_G$, which is thus not isolated.
    \end{proof}

\begin{corollary}[R. Pavlov, S. Schmieding]\label{cor pavlov}
A one-dimensional shift is isolated in $\mathcal{S}^1$ if and only if it has the no-middle-cycle property.
\end{corollary}

\begin{proof}
    Consider a one-dimensional shift $X$ and for every integer $n$, denote by $X_n$ the shift $X_{G_{X,n}}$
    % the shift on the same alphabet $\mathcal{A}$ defined by forbidding words in $\mathcal{A}^{\llbracket0,n-1\rrbracket}\setminus\mathscr{L}_{\llbracket 0, n-1 \rrbracket}(X)$. We have $X_n \rightarrow X$ 
    and we apply Lemma~\ref{lemma.nmc}.

   $(\Rightarrow)$ If $X$ does not satisfy the no-middle-cycle property, then since each $X_n$ is not isolated, $X$ is also not isolated.
   
    $(\Leftarrow)$ If $X$ has the no-middle-cycle property, then one of the shifts $X_n$ has this property as well, and thus $X_n$ has finitely many subsystems (Lemma~\ref{lemma.nmc}). Since $X\subset X_n$, {we can apply Lemma \ref{lemma:subsystem.finite.subsystems} and conclude that $X$ is isolated.}
\end{proof}

\subsection{Finitely many subsystems are not necessary\label{section.example.inf.subsystems}}

If a shift has finitely many subsystems, then it satisfies the $(\star)$ property. One could wonder if the converse is true. Although it is the case for $d=1$ (Lemma~\ref{lemma.nmc}), we prove here that it is false in the case when $d \ge 2$. \bigskip 

We provide here an example of shift of finite type isolated in $\mathcal{S}^2$ which has infinitely many subsystems. This example can be generalized straightforwardly to all the spaces $\mathcal{S}^d$, $d \ge 2$. A formal proof of this statement would be quite involved due to the elements used in its construction. Since it is not essential to the purpose of the present paper, we only present here elements of the construction and the main arguments. We use the terminology introduced in Section~\ref{section.constructions}. \bigskip 

% \comm{A}{et si \`a la place de prop on l'appelle ``exemple'' ? comme \c{c}a on peut metre une prueve plus r\'elax\'ee sans avoir la n\'ecessit\'e de poser des milliards des excusses :')}

% \comm{S}{Oui je suis d'accord}

% \begin{proposition}\label{ex robi}
% There exists a shift of finite type isolated in $\mathcal{S}^2$ which has infinitely many subsystems.
% \end{proposition}

We construct a shift of finite type $X$ which has a unique maximal subsystem, infinitely many subsystems, and such that all of its strict subsystems are included in the maximal subsystem. This shift $X$ is then isolated. We construct this shift by adding several layers, one by one, to the shift $X_R^{+}$. These layers are as follows:

\begin{enumerate}
\item We refer to the first of these layers as the color layer. The elements of its alphabet are red and blue. For all cells the color is constant over the union of its border and its cytoplasm. If a cell has a border, then this color is blue.
\item The second layer has alphabet $\{0,1,2,\square\}$ and is described as follows: for each cell not included 
in a cell whose cytoplasm or border are colored in blue (for simplicity we call this a \textit{maximal} cell), the positions in its border which have a neighbor in a free column 
or a free row of this cell 
are superimposed with a 
symbol in $\{0,1,2\}$, and every other position is superimposed with $\square$.
We add an accessory layer 
which makes all the sequences of symbols written on an edge of the border equal (read from left to right for the vertical edges and from top to bottom for horizontal ones) and makes this common word on letters $0,1$ only or on $2$ only. 
%For all $n$, there is a cyclic permutation $\textbf{r}_n$
%of the set of words which can be read in the present layer on the edge of an order $n$ maximal cell such that it is possible to 
%add another accessory layer 
%which ensures that for all $n$, and two maximal cells of order $n$ which are adjacent, on the one on the right (resp. top) the word is $\textbf{r}_n(w)$, 
%where $w$ is the word 
%of the other one }
An accessory layer is added such that for any maximal cell of order $n$ whose border reads $w$ the border of any adjacent order $n$ maximal cell reads $r_n(w)$, where $r_n$ is some fixed cyclic permutation (for a more precise description of this, see~\cite{Gangloff}). The set of possible words $w$ has cardinal $p_n := 2^{2^n}+1$ and we can associate this set with $\mathbb{Z}/p_n\mathbb{Z}$ 
in such a way that $r_n(w) = r_n(w)+1$ for all $w$.

\item The third layer is on alphabet $\{0,1,\square\}$. 
A position is superimposed with a symbol in $\{0,1\}$ if and only if it is in the border of a cell strictly contained in a maximal cell.
This symbol is constant over the border. 
%\st{If we denote by $n$ the order of this maximal cell, the north west order $n-1$ cell inside it has symbol $1$ and the other ones $0$. For all $k < n-1$, there is only one order $k$ cell which is superimposed with $1$, and it is the one which is on the north west of an order $k+1$ cell superimposed with $1$.} 
Inside an order $n$ maximal cell, for any $k\le n$ it holds that the symbol on the border of any order $k-1$ cell is 1 if and only if it is the north west cell within cell which has order $k$ and border symbol 1 or has order $n$.
\item The fourth layer is also on alphabet $\{0,1,\square\}$. A position is superimposed with a symbol in $\{0,1\}$ if and only if it is in a free column of a maximal cell. This symbol is constant over a free column. On the extremities of a free column, the symbol in the current layer is equal to the one on the border position adjacent to it in Layer \textbf{2}, or $0$ if this symbol is a $2$.
%\textcolor{red}{1 if the symbol in Layer 2 is 1, and 0 otherwise.}
\item We add a last layer with alphabet $\{0,1,\square\}$ and defined as follows: a position has a symbol in $\{0,1\}$ if and only if 
it is in the border of a cell strictly contained in a maximal cell. 
%\st{For all $n$ symbol is the same for all the positions which are in the border of a cell of order $n$ included in the same maximal cell. For all 
%position which is marked with $1$ in layer \textbf{3}, if it has a neighbor which is marked with a symbol in $\{0,1\}$ 
%in layer \textbf{4}, then the symbol in the current layer is equal to this one.} 
Within a maximal cell, order $n$ cells all have the same symbol on their borders. In positions where Layer \textbf{3} has a 1 and one of the neighboring positions has either a 0 or a 1 on Layer \textbf{4}, the same symbol (as in Layer \textbf{4}) is placed in the current layer.
\end{enumerate}

With the aid of accessory layers we make $X$ a shift of finite type, using a similar strategy as the one used to make $X_R^+$ a shift of finite type -- accessory layers ``transmit'' information to force the desired patterns.

Consider the subsystem $Z$ of $X$ which consists in configurations in which the red color does not appear. %\textcolor{red}{Maybe explain in a bit more detail the whys of the claims that follow. I think the description of the construction is more or less clear, but the conclusions sounds kinda magical, they just appear. For instance the relation between being having a dense orbit and Fermat's numbers is super unclear.} 
It is maximal because any configuration with the red color has a dense orbit. In order to prove this, consider any pattern $p$ in the language of the shift $X$, and any configuration $x \in X$ in which the red color appears, and prove that $p$ appears in $x$. There are two possibilities: the color red appears in $p$ or not. Let us consider these two cases separately:

\begin{itemize}
    \item \textbf{If the color red appears in $p$.} In this case, $p$ can be completed into a pattern $q$ over a supertile centered on a cell in which the rules of $X$ are not broken. Within the set of patterns on this supertile, $q$ is determined by the sequence $w_0(q) , \ldots , w_k(q)$, where $k$ is the largest order of a cell in the supertile, and $w_i(q)$ is the word in Layer \textbf{2} on the top border of the rightmost topmost cell of order $i$ in the supertile. This supertile is of order $2k+1$ by definition, and it is known that it appears periodically in $x$ in both direction with period $4^{2k+1}$. Furthermore, because of the last accessory layer described in point \textbf{2} above, for two adjacent of these supertiles which support a maximal order $k$ cell, denoting by $q'$ and $q''$ the patterns appearing in $x$ on the left one and the right one respectively, the sequence $(w_0(q''), \ldots , w_k(q''))$ is the image of $(w_0(q'), \ldots , w_k(q'))$ by the following function from $\mathbb{Z}/p_0 \mathbb{Z} \times \ldots \times \mathbb{Z}/p_k \mathbb{Z}$ to itself:
    \[(w_0 , \ldots , w_k) \mapsto (w_0 + 4^{2k+1} , \ldots , w_k +1).\]
    Because the numbers are all co-prime (Goldbach’s theorem on Fermat numbers) and co-prime with $2$, this is a cyclic permutation. Because one can find in $x$ arbitrarily long sequences of adjacent supertiles as above, $q$ appears in $x$ and so does $p$.
    \item \textbf{If the color red does not appear in $p$.} Then $p$ can be completed into a pattern $q$ over a cell in which the rules of $X$ are not broken. Among all the valid patterns over the same cell, $q$ is determined by the word on the top part of the border in Layer \textbf{2}. Within the configuration $x$, this word takes all the possible values, because of the last accessory layer described in point \textbf{2} above. Therefore $q$ appears in $x$ and so does $p$.
\end{itemize}

We just have proved that all the configurations in which the red color appear have a dense orbit in $X$. For this reason every subsystem of $X$ is included in $Z$ which is the only maximal subsystem. As a consequence of Theorem~\ref{thm.main}, $X$ is isolated.

Furthermore, it is straightforward that $Z$ has infinitely many subsystems, because the set of configurations which hold a certain fixed sequence of Layer \textbf{3} symbols $\{0,1\}$ on cells of order $n$ is a shift and that for two different sequences, these two shifts are disjoint. Moreover, for every sequence this shift is not empty. Thus, $X$ has infinitely many subsystems.
% \comm{A}{see comment 44: dire que les sousystemes sont differents, preciser que les symboles sont ceux de la couche 3}

\subsection{Properties of shifts in relation with maximal subsystems\label{section.study.maximal}}

In this section we address some natural questions about maximal subsystems and the $(\star)$ property.

\begin{definition}
    For all shift $X$, we refer to the cardinality of $\mathcal{M}(X)$ as the \textbf{maximality type} of $X$.
\end{definition}

We will see in this section that every element of $\mathbb{N} \cup \{\infty\}$ is the maximality type of at least one shift [\cref{maxtype.realisation}].
%Observe that a minimal shift has maximality type 0 as it has no subsystem. The maximality type of a full shift is 0 as well, although it has infinitely many subsystems, but none of them is maximal (\cref{example.maxtype.fullshift}). \cref{maxtype.realisation} shows that for any $n\in\N\cup\{\infty\}$ there exists a shift with maximality type $n$.

% \comm{S}{Why this remark here since we see these things later ?}
% \comm{A}{In (45) it is asked to specify the range of possible values for the maximality type}

% \comm{S}{I think maybe it is a bit unclear for the reviewer. Maybe we should say that the maximality type is the cardinality of the set of maximal subsystems, and that in principle it can take values from 0 to infinity included, maybe just say that we prove later that all these values are possible (without necessarily citing the results).}
% \comm{A}{I don't see a major issue to cite the results ahead as some sort of introduction or guideline for what's coming.}

\paragraph{{The maximality type is invariant under conjugacy}}

Observe that having a conjugacy between two sets amounts, in particular, to having a bijection of subsystems that preserves the inclusion relation. Thus, if $X$ and $Y$ are conjugates, the maximal subsystems of $X$ (if any) are sent to the maximal subsystems of $Y$, and vice versa. This observation yields the following.

\begin{proposition}
    All conjugates of a shift have the same maximality type.
\end{proposition}

%\begin{proof}
%Consider two shifts $X,X'$ and $f : X \rightarrow X'$ be a conjugacy. If $Z$ is a maximal subsystem of $X$ then $f(Z)$ is a maximal subsystem of $X'$. Indeed, it is not equal to $X'$ otherwise $Z=X$, and if $f(Z) \subset Z' \subset X'$, then $Z \subset f^{-1}(Z') \subset X$. Since $Z$ is maximal, $f^{-1}(Z') = Z$ or $f^{-1}(Z') = X$, thus $Z' = f(Z)$ or $Z' = X'$. Furthermore, let $m_f : \mathcal{M}(X) \rightarrow \mathcal{M}(X')$ (resp. $m_{f^{-1}} : \mathcal{M}(X') \rightarrow \mathcal{M}(X)$) be a map such that $m_f(Z) = f(Z)$ (resp. $m_{f^{-1}}(Z) = f^{-1}(Z)$), then  $m_f \circ m_{f^{-1}} = m_{f^{-1}} \circ m_f$ is the identity. This implies the statement.
%\end{proof}

\paragraph{{Realization of all numbers of maximal subsystems}\label{section.type.realization}}

%\begin{notation}
%For a shift $X$ of maximality type $1$, we denote by $M(X)$ the sole maximal subsystem of $X$ and by $R(X)$ the elements of $X$ which are not in $M(X)$.
%\end{notation}

\begin{lemma}\label{lemma.inf.max.subsystems}
Let $X$ be a shift which is the union of shifts $Z_n = K \sqcup M_n $, $n\ge 0$, where $K$ is a shift and for all $n\ge 0$, $M_n$ is a minimal shift such that:
\begin{itemize}
\item $(M_m = M_n) \Rightarrow (m=n)$;
\item $Z_n \rightarrow K$.
\end{itemize}
Then the maximal subsystems of $X$ are the shifts $\bigcup_{n \neq m} Z_n$, $m \ge 0$. 

\end{lemma}
% \comm{A}{le vide n'est pas un sous-systeme}

\begin{proof}
Fix some integer $m$.
The set $L_m := \left(\bigcup_{n \neq m} Z_n \right)$ is a shift since it is shift-invariant and closed because $Z_n \rightarrow K = \bigcap_m Z_m$. %Since $Z_n \rightarrow K$, there exists $n_0 > m$ such that for all $n \ge n_0$, $\delta_H (M_n,K) < \delta(M_m,K)$. This means that $M_m$ is disjoint from $K \cup \left(\bigcup_{n>n_0} M_n\right)$, and thus it is disjoint from $L_m$. 
Observe that $L_m$ is disjoint from $M_m$, and as a consequence, $L_m = X \backslash M_m$, which makes this shift a maximal subsystem of $X$ (since $M_n$ is minimal). %Since $Z_n \rightarrow K$, 
Every proper subsystem of $X$ is contained in some $L_m$, which means that the maximal subsystems of $X$ are exactly the sets $L_m$, $m\ge 0$.
% Moreover, it contains $X \backslash R(Z_m)$: indeed, if $x \notin R(Z_m)$, then it is in the unique maximal subsystem of $Z_m$ or in some $Z_n$, $n \neq m$, and thus in $\left(\bigcup_{n \neq m} Z_n \right)$. Furthermore, 
% this set is also contained in $X \backslash R(Z_m)$. Indeed, if an element of $R(Z_m)$ were in $\left(\bigcup_{n \neq m} Z_n \right)$, it would be in some $Z_n, n \neq m$. Then $Z_m$ would be contained in $Z_n$ and since $Z_n$ contains the unique maximal subsystem of $Z_n$, and that this subsystem is also a maximal subsystem of $Z_m$, this would imply that $Z_m = Z_n$, which is false by hypothesis. 
% We have proved that 
% $\left(\bigcup_{n \neq m} Z_n \right) = X \backslash R(Z_m)$. This implies that $\left(\bigcup_{n \neq m} Z_n \right)$ is maximal, because 
% any shift strictly containing this set would contain 
% $\left(\bigcup_{n} Z_n \right) = X$, and thus be equal to $X$.
\end{proof}

\begin{example}\label{example.inf.max.sub}
    
    Consider the following alphabet
    : 
    
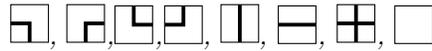
\begin{center}
        \begin{tikzpicture}[scale=0.5] \draw[line width=0.5mm] (0.5,0) -- (0.5,0.5) -- (0,0.5);\draw (0,0) rectangle (1,1);\end{tikzpicture},
        \begin{tikzpicture}[scale=0.5] \draw[line width=0.5mm] (0.5,0) -- (0.5,0.5) -- (1,0.5);\draw (0,0) rectangle (1,1);\end{tikzpicture},\begin{tikzpicture}[scale=0.5] \draw[line width=0.5mm] (1,0.5) -- (0.5,0.5) -- (0.5,1);\draw (0,0) rectangle (1,1);\end{tikzpicture},\begin{tikzpicture}[scale=0.5] \draw[line width=0.5mm] (0,0.5) -- (0.5,0.5) -- (0.5,1);\draw (0,0) rectangle (1,1);\end{tikzpicture},
        \begin{tikzpicture}[scale=0.5] \draw[line width=0.5mm] (0.5,0) -- (0.5,1);\draw (0,0) rectangle (1,1);\end{tikzpicture},
        \begin{tikzpicture}[scale=0.5] \draw[line width=0.5mm] (0,0.5) -- (1,0.5);\draw (0,0) rectangle (1,1);\end{tikzpicture},
        \begin{tikzpicture}[scale=0.5] \draw[line width=0.5mm] (0,0.5) -- (1,0.5);\draw[line width=0.5mm] (0.5,0) -- (0.5,1);\draw (0,0) rectangle (1,1);\end{tikzpicture},
        \begin{tikzpicture}[scale=0.5]\draw (0,0) rectangle (1,1);\end{tikzpicture}
    \end{center}
Define the two-dimensional configuration $x_+$ made only with the cross symbol, and for all $n\in\mathbb{N}$ define the two-dimensional configurations $x_n$ as follows:

\begin{enumerate}
\item The cross symbol$\ $ \begin{tikzpicture}[scale=0.35,baseline=0.8mm] \draw[line width=0.5mm] (0,0.5) -- (1,0.5);\draw[line width=0.5mm] (0.5,0) -- (0.5,1);\draw (0,0) rectangle (1,1);\end{tikzpicture} appears in
positions in $ \mathbb{N}_{n} :=(2n+1)\mathbb{Z}^2$;

\item The vertical symbol$\ $ \begin{tikzpicture}[scale=0.35,baseline=0.8mm] \draw[line width=0.5mm] (0.5,0) -- (0.5,1);\draw (0,0) rectangle (1,1);\end{tikzpicture}  appears in positions in $$\mathbb{G}^v_{n}:= \left\{((2n+1)k,l) : k,l \in \mathbb{Z} \right\}\backslash \mathbb{N}_{n};$$

\item The horizontal symbol$\ $ \begin{tikzpicture}[scale=0.35,baseline=0.8mm] \draw[line width=0.5mm] (0,0.5) -- (1,0.5);\draw (0,0) rectangle (1,1);\end{tikzpicture} appears in position in $$\mathbb{G}^h_{n}:= \left\{(k,(2n+1)l) : k,l \in \mathbb{Z} \right\}\backslash \mathbb{N}_{n};$$
\item On each of the translates of $\mathbb{U}_{2(n-1)}^2$ which do not intersect $\mathbb{G}^v_{n}$, 
$\mathbb{G}^h_{n,}$ nor
$\mathbb{N}_{n,}$, the pattern is of the following form: 

\begin{center}
    \begin{tikzpicture}[scale=0.5] 
    \draw (0,0) grid (8,8);
    \foreach \x in {0,1,2,3} {
    \draw[line width=0.5mm] (\x + 0.5,\x + 0.5) rectangle (7.5-\x,7.5-\x);
    }
        
    \end{tikzpicture}
\end{center}
\end{enumerate}
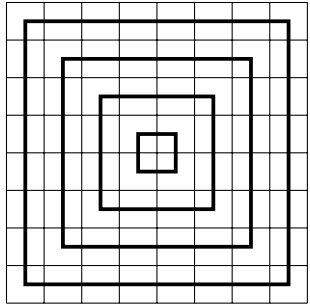
Figure~\ref{fig.config.ex} illustrates this definition. Define the two-dimensional shift $X$ as the closure of the following set: 

\[\mathcal{O}(x_+) \cup \left(\bigcup_n \mathcal{O}(x_n)\right).\]

\begin{figure}[H]
\centering
\begin{subfigure}[t]{.32\textwidth}\centering
\begin{tikzpicture}[scale=0.5] 
\draw (0,0) grid (7,7);
\foreach \x in {0,3,6}{
    \draw[line width=.5mm](\x+.5,0) -- (\x+.5,7);
    \draw[line width=.5mm](0,\x+.5) -- (7,\x+.5);}
\foreach \x in {1,4}{\foreach \y in{1,4}{
    \draw[line width=.5mm](\x+.5,\y+.5) rectangle (\x+1.5,\y+1.5);}}    
\end{tikzpicture}
\subcaption{Configuration $x_1$}
\end{subfigure}
\vspace{.05\textwidth}
\begin{subfigure}[t]{.32\textwidth}\centering
\begin{tikzpicture}[scale=.435]
\draw(0,0) grid (8,8);
\foreach \x in {0,7}{
    \draw[line width=.5mm](\x+.5,0) -- (\x+.5,8);
    \draw[line width=.5mm](0,\x+.5) -- (8,\x+.5);}
\foreach \x in {1,2,3}{
    \draw[line width=.5mm](4-\x+.5,4-\x+.5) rectangle (4+\x-.5,4+\x-.5);}
\end{tikzpicture}
\subcaption{Configuration $x_3$}
\end{subfigure}
\vspace{.05\textwidth}
\begin{subfigure}[t]{.32\textwidth}\centering
\begin{tikzpicture}[scale=.5]
\draw(0,0) grid (7,7);
\foreach \x in {0,...,6}{
    \draw[line width=.5mm](\x+.5,0) -- (\x+.5,7);
    \draw[line width=.5mm](0,\x+.5) -- (7,\x+.5);}
\end{tikzpicture}
\subcaption{Configuration $x_+$}
\end{subfigure}
\caption{Different configurations of $X$.\label{fig.config.ex}}
\end{figure}
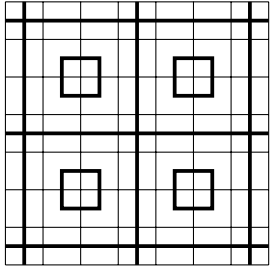
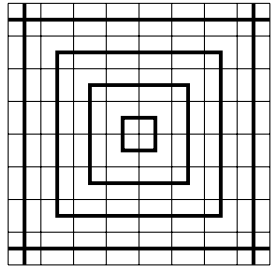
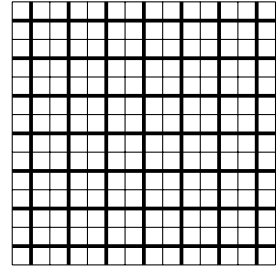

%%%%%%%%%%%%%%%%%%%%%%%%%%%%%%%%%%%%%%
%%%%%%%%%%%%%%%%%%%%%%%%%%%%%%%%%%%%%%
%%%%%%%%%%%%%%%%%%%%%%%%%%%%%%%%%%%%%%

Observe that in these configurations the occurrence of a corner symbol$\ $\begin{tikzpicture}[scale=0.35,baseline=0.8mm] \draw[line width=0.5mm] (0.5,0) -- (.5,0.5) -- (1,.5);\draw (0,0) rectangle (1,1);\end{tikzpicture} in position $(i,j)$ forces that at position $(i-1,j+1)$ either the same corner symbol or a cross symbol occurs. The same holds for the rotations
$\ $\begin{tikzpicture}[scale=0.35,baseline=0.8mm] \draw[line width=0.5mm] (0.5,0) -- (.5,0.5) -- (0,.5);\draw (0,0) rectangle (1,1);\end{tikzpicture},
$\ $\begin{tikzpicture}[scale=0.35,baseline=0.8mm] \draw[line width=0.5mm] (0.5,1) -- (.5,0.5) -- (0,.5);\draw (0,0) rectangle (1,1);\end{tikzpicture}, and
$\ $\begin{tikzpicture}[scale=0.35,baseline=0.8mm] \draw[line width=0.5mm] (0.5,1) -- (.5,0.5) -- (1,.5);\draw (0,0) rectangle (1,1);\end{tikzpicture}, 
of$\ $\begin{tikzpicture}[scale=0.35,baseline=0.8mm] \draw[line width=0.5mm] (0.5,0) -- (.5,0.5) -- (1,.5);\draw (0,0) rectangle (1,1);\end{tikzpicture}. Using this, it is possible to define $\mathcal{F}=\mathcal{A}^{\llbracket0,1\rrbracket\times\llbracket0,1\rrbracket}\setminus\mathcal{P}$, where the set $\mathcal{P}$ consists of the following $2\times2$ patterns
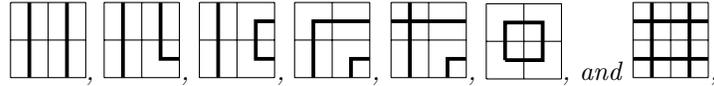
\begin{center}
\begin{tikzpicture}[scale=0.5] \draw[line width=0.5mm] (0.5,0) -- (0.5,2);
\draw[line width=0.5mm] (1.5,0) -- (1.5,2);
\draw (0,0) grid (2,2);
\end{tikzpicture},
\begin{tikzpicture}[scale=0.5] \draw[line width=0.5mm] (0.5,0) -- (0.5,2);
\draw[line width=0.5mm] (2,.5) -- (1.5,0.5) -- (1.5,2);
\draw (0,0) grid (2,2);
\end{tikzpicture},
\begin{tikzpicture}[scale=0.5] \draw[line width=0.5mm] (0.5,0) -- (0.5,2);
\draw[line width=0.5mm] (2,.5) -- (1.5,0.5) -- (1.5,1.5) -- (2,1.5);
\draw (0,0) grid (2,2);
\end{tikzpicture},
\begin{tikzpicture}[scale=0.5] \draw[line width=0.5mm] (0.5,0) -- (0.5,1.5) -- (2,1.5);
\draw[line width=0.5mm] (1.5,0) -- (1.5,.5) -- (2,.5);
\draw (0,0) grid (2,2);
\end{tikzpicture}, 
\begin{tikzpicture}[scale=0.5] \draw[line width=0.5mm] (0.5,0) -- (0.5,2);
\draw[line width=.5mm] (0,1.5) -- (2,1.5);
\draw[line width=0.5mm] (1.5,0) -- (1.5,.5) -- (2,.5);
\draw (0,0) grid (2,2);
\end{tikzpicture}, 
\begin{tikzpicture}[scale=0.5] \draw[line width=0.5mm] (0.5,0.5) -- (0.5,1.5) -- (1.5,1.5) -- (1.5,.5) -- (.5,.5);
\draw (0,0) grid (2,2);
\end{tikzpicture}, and 
\begin{tikzpicture}[scale=0.5] \draw[line width=0.5mm] (0.5,0) -- (0.5,2);
\draw[line width=0.5mm] (1.5,0) -- (1.5,2);
\draw[line width=0.5mm] (0,0.5) -- (2,0.5);
\draw[line width=0.5mm] (0,1.5) -- (2,1.5);
\draw (0,0) grid (2,2);
\end{tikzpicture}, 
\end{center}
and their rotations. We leave it to the reader to check that $X=X_{\mathcal{F}}$, which means that the shift $X$ is of finite type. 

%%%%%%%%%%%%%%%%%%%%%%%%%%%%%%%%%%%%%%
%%%%%%%%%%%%%%%%%%%%%%%%%%%%%%%%%%%%%%
%%%%%%%%%%%%%%%%%%%%%%%%%%%%%%%%%%%%%%

Denote by $K :=X \backslash \bigcup_n \mathcal{O}(x_n)$. Then $X$ satisfies the conditions of Lemma~\ref{lemma.inf.max.subsystems},
where $Z_n := \mathcal{O}(x_n) \cup K$ for all $n$. 

% is the orbit of a configuration in which the cross symbol appears periodically with period $2n$.
% \comm{A}{Selon le reviewer (obs 48) $\cap_mZ_m=\emptyset$ (je suis d'accord), mais $Z_n\not\to\emptyset$, ce qui ne safisfait pas les conditions du \cref{lemma.inf.max.subsystems}, je suis d'accord avec \c{c}a aussi. Je pense qu'il faut prendre $Z_n=\mathcal{O}(x_n)\cup\mathcal{O}(x_\infty)$, o\`u $x_\infty$ es la configuration qu'a que des carr\'es (comme celui l\`a au dessin au point 4), avec \c{c}a je pense qu'on a $\cap_mZ_m=\mathcal{O}(x_\infty)$ et $Z_n\to\mathcal{O}(x_\infty)$}
% \comm{Silvere}{Pour ta suggestion, je suis d'accord mais il faut remplacer l'orbite de $x_{\infty}$ par toutes les configurations qui ne sont pas un shift de $x_n$ pour un $n$ ou de $x_{+}$.}
\end{example}

\begin{proposition}\label{maxtype.realisation}
For all $n \in \mathbb{N} \cup \{+\infty\}$ and all $d > 1$, there is a $d$-dimensional shift of finite type which has exactly $n$ maximal subsystems.
\end{proposition}

\begin{proof}
The shift presented in example~\ref{example.inf.max.sub} has infinitely many maximal subsystems, because it satisfies the conditions of Lemma~\ref{lemma.inf.max.subsystems}. For any integer $n$, it is sufficient to see that the disjoint union of $n$ shifts, each one consisting in a single periodic orbit, is of finite type and has exactly $n$ maximal subsystems.
\end{proof}

\paragraph{{Maximality and minimality}}

It is natural to wonder whether there is a relation between maximal and minimal subsystems.  We present two results in this direction.

%\begin{proposition}
%    Let us consider $X$ a shift of finite type which has at least two maximal subsystems. Then for all such subsystem $Z$, 
%    $\overline{X \backslash Z}$ is a minimal shift. Therefore, $X$ is equal to the union of some minimal shifts with the intersection of its maximal subsystems.
%\end{proposition}

\begin{proposition}\label{proposition.minimal}
    If a shift $X$ has a subsystem $Z$ such that $\overline{X \backslash Z}$ is a minimal shift, then $Z$ is a maximal subsystem of $X$.
\end{proposition}

\begin{proof}
Indeed, consider any point $x \in X \backslash Z$. 
Since $\overline{\mathcal{O}(x)} = \overline{X \backslash Z}$, then $Z \bigcup \overline{\mathcal{O}(x)} = X$. Therefore, all the shifts which contain $Z$ are equal to $Z$ or $X$.
\end{proof}

\begin{notation}
For a shift $X$ of maximality type $1$, we denote by $M(X)$ the sole maximal subsystem of $X$ and by $R(X)$ the elements 
of $X$ which are not in $M(X)$.
\end{notation}

\begin{proposition}
Provided a shift $X$ and $Z$ 
a subsystem of $X$, if every strict subsystem of $Z$ is included in $\overline{X \backslash Z}$, then $\overline{X \backslash Z}$ 
is either a maximal subsystem of $X$ or equal to $X$.
\end{proposition}

This is the case if $Z$ is a minimal subsystem of $X$ for instance. 

\begin{proof}
If $Z$ is %different from 
equal to the closure of the union of the elements of $\mathcal{S}(Z)$, then  $\overline{X \backslash Z}$ must contain $Z$ and is equal to $X$. Otherwise, $Z$ is necessarily of maximality type $1$. The closure of $X\backslash Z$ contains the maximal subsystem of $Z$ by hypothesis. So in any case $X \backslash (\overline{X \backslash Z})\subset R(Z)$. If the closure of $X \backslash Z$ is not $X$, then the orbit of any element outside of it is dense in $Z$. This implies that $\overline{X \backslash Z}$ is a maximal subsystem of $X$.
%f $\overline{X \backslash Z}$ is not equal to $X$, then $X \backslash \left(\overline{X \backslash Z} \right)\subset R(Z)$. This implies that $\overline{X \backslash Z}$ is maximal.
\end{proof}

\begin{remark}
In particular if $X$ has maximality type 0, 
then for every strict subsystem $Z$, if $\overline{X \backslash Z} \neq X$,
there exists another subsystem $Z'$ such that $Z' \not \subset \overline{X \backslash Z}$.
\end{remark}

%\begin{remark}
%    Let us notice that there exist subshifts of finite type with infinitely many maximal subsystems: for instance 
%    if $X$ is the minimal Robinson superimposed with bits on uniform on every level; or even the non-minimal Robinson. The common property of these two examples is that they can be written as $Z \bigcup \left(\bigcup_n Z_n \right)$, where the sushifts $Z_n$ are disjoint and converge to $Z$, which is as well disjoint from all the $Z_n$. For all $m$, the subshift $Z \bigcup \left(\bigcup_{n \neq m} Z_n \right)$ is maximal. 
%\end{remark}

\paragraph{{Maximal subsystems of a shift of finite type are of finite type}}

\begin{proposition}
    Consider any set of patterns $\mathcal{F}$ such that $X:= X_{\mathcal{F}} \neq \emptyset$. If $Z$ is a maximal subsystem of $X$, 
    then there exists a pattern $p \notin \mathcal{F}$ such that 
    $Z = X_{\mathcal{F} \cup \{p\}}$.
\end{proposition}

\begin{proof}
    Consider the set $\mathcal{P}$ of patterns on the alphabet of 
    $X$ which are not in the language of $Z$. This set must contain $\mathcal{F}$. 
    Since $Z$ is different from $X$, it cannot be equal to $\mathcal{F}$. Let us proceed ad absurdum, and suppose that for all $p \in \mathcal{P} \backslash \mathcal{F}$, $X_{\mathcal{F} \cup \{p\}} = X$. Then for all finite subset $\mathcal{Q} \subset \mathcal{P}$ which contains $\mathcal{F}$, $X_{\mathcal{Q}} = X_{\mathcal{F}}$. 
    Set $(p_n)_{n \in \mathbb{N}}$ an enumeration of the elements of $\mathcal{P}$, $\mathcal{P}_n :=\{p_1,\ldots,p_n\}\cup\mathcal{F}$ and $X_n :=X_{\mathcal{P}_n}$.
    Therefore, since $X_{n}$ converges to $Z$ when $n \rightarrow \infty$, we have $Z = X$, which is impossible. Thus there exists $p \in \mathcal{P} \backslash \mathcal{F}$ such that $X_{\mathcal{F} \cup \{p\}} \neq X$. Since $Z \subset X_{\mathcal{F} \cup \{p\}}$ and $Z$ is maximal, we have 
    $Z = X_{\mathcal{F} \cup \{p\}}$.
\end{proof}

% \begin{proof}[Proof by the reviewer]
% \textcolor{cyan}{Let $\mathcal{P}=\{p_1,p_2,\ldots\}$ be an enumeration of all patterns on the alphabet of $X$ that do not appear in the language of $Z$. Set $\mathcal{P}_n=\{p_1,\ldots,p_n\}\cup\mathcal{F}$ and $X_n=X_{\mathcal{P}_n}$. For each $n\in\N$ have that $Z\subset X_n\subset X$, and by the maximality of $Z$, we must have either $X_n=Z$ or $X_n=X$. Since $X_n\to Z$ (by \cref{lemma.intersection}), there exists $n\in\N$ such that $X_n=Z$. Choose the minimal such $n$ and note that $X_{\mathcal{F}\cup\{p_n\}}=X_n=Z$.}
% \end{proof}

% \comm{S}{This proof by the reviewer does not work. We just need to specify what $\mathcal{Q} \rightarrow \mathcal{P}$ means (which is so clear from context but well).}

\begin{corollary}
    If $X$ is a shift of finite type, each of its maximal subsystems is also of finite type.
\end{corollary}

\begin{corollary}\label{corollary.countable}
For every shift $X$, $\mathcal{M}(X)$ is countable.
\end{corollary}

\paragraph{{Additivity of types relatively to disjoint union}}

The following is straightforward, but we include a proof in order to illustrate how maximal subsystems of a disjoint union arise.

\begin{proposition}
Let $X,Z$ be non-minimal shifts with empty intersection, and let $m,n$ be the maximality type of $X,Z$, respectively. Then the maximality type of $X \cup Z$ is $n+m$. If $X$ is minimal and $Z$ is not minimal, $X \cup Z$ is of type $n+1$. If both $X$ and $Z$ are minimal, $X \cup Z$ has type $2$.
\end{proposition}

\begin{proof}
     Let us consider the first case. It is sufficient to see that given two disjoint shifts $X$ and $Z$, the subsystems of $X \cup Z$ of the following form are maximal: i) $X\cup M$, where $M$ is a maximal subsystem of $Z$ if $Z$ is not minimal and equal to $Z$ otherwise; ii) $N\cup Z$, where $N$ is a maximal subsystem of $X$ if $X$ is not minimal and equal to $X$ otherwise. Any other strict subsystem of $X \cup Z$ is of the form $N\cup M$, where $N$ (resp. $M$) is a subsystem of $X$ (resp. $Z$) or empty. Thus, such a subsystem is included in a subsystem as described above. Hence, they are the only maximal subsystems.
\end{proof}

\section{\label{section.dec}Maximality decomposition}

We prove in this section that every non-minimal shift can be written as the union of two parts: a union of shifts which are transitive shifts of maximality type $1$ or minimal; a set which is empty or a non-minimal shift having maximality type $0$ (Section~\ref{section.max.dec}). We then prove that this decomposition is unique in a certain sense (Section~\ref{section.other.dec}). This will be useful in Section~\ref{section.cantor.bendixson}.

\subsection{\label{section.max.dec}Definition and properties}

\begin{lemma}\label{lemma:maxtype}
    Consider a shift $X$ and $Z$ a maximal subsystem of $X$. If $Z \cap \overline{X \backslash Z} = \emptyset$, then $\overline{X \backslash Z}$ is minimal. Otherwise, 
    $\overline{X \backslash Z}$ is 
    transitive of maximality type $1$.
\end{lemma}

\begin{proof}
    Since $Z$ is maximal, then for all $x \in X \backslash Z$, 
    $\overline{\mathcal{O}(x)}$ contains $X \backslash Z$ and since it is closed, it also contains $\overline{X \backslash Z}$. In the first case, $X \backslash Z$ must be closed, and the argument above implies that it is minimal. In the second case, $Z \cap \overline{X \backslash Z}$ is the unique maximal subsystem of $\overline{X \backslash Z}$ and $\overline{X \backslash Z}$ is transitive.
\end{proof}

\begin{notation}
For every shift $X$, we  set 
$\mathcal{T}(X) := \left\{ \overline{X \backslash Z} : Z \in \mathcal{M}(X)\right\}$,  $\mathcal{E}(X):= \overline{\bigcup_{T \in \mathcal{T}(X)} T} $ and $\mathcal{K}(X):= \overline{X \backslash \mathcal{E}(X)}$.
\end{notation}

\begin{remark}
In the particular case that $X$ is minimal, $\mathcal{T}(X) = \mathcal{E}(X) = \emptyset$ and $\mathcal{K}(X) = X$.
\end{remark}

\begin{lemma}\label{lemma.max.dec}
    For every shift $X$, $X = \mathcal{K}(X) \cup \mathcal{E}(X)$. We will refer to this equality as the \textit{maximality decomposition} of $X$. Furthermore, $\mathcal{K}(X)$ is a subset of 
 $\bigcap_{Z \in \mathcal{M}(X)} Z$.
\end{lemma}

\begin{proof}
    The first assertion is a direct consequence of the definition of $\mathcal{K}(X)$. The second one comes from the fact that for all maximal subsystem $Z$ of $X$, $X \backslash (\overline{X \backslash Z})$ is a subset of $Z$.
\end{proof}

\begin{lemma}\label{lemma.non.inclusion}
    For every shift $X$ and $Z \in \mathcal{M}(X)$, we have: \[ \overline{X \backslash Z} \subset \bigcap_{\underset{Z' \neq Z}{Z' \in \mathcal{M}(X)}} Z'\]
\end{lemma}

\begin{proof}
    This is a direct consequence of Lemma~\ref{lemma.two.maximal}.
\end{proof}

The following is straightforward: 

\begin{lemma}\label{lemma.nucleus}
    Consider a shift $X$ and $K$ another shift such that $X = K \cup \mathcal{E}(X)$. Then $\mathcal{K}(X) \subset K$.
\end{lemma}

\begin{lemma}
    Whenever $X$ is not minimal, if $\mathcal{K}(X)$ is not empty, it is a shift which has maximality type $0$ and is not minimal.
\end{lemma}

\begin{proof}
    Assume that $\mathcal{K}(X)$ is not empty. Let us assume ad absurdum that $\mathcal{K}(X)$ has at least one maximal subsystem.
    For all $Z \in \mathcal{M}(\mathcal{K}(X))$, 
    $\mathcal{K}(X) \backslash Z$ is included in $\mathcal{E}(X)$. Otherwise $Z \cup \mathcal{E}(X)$ would be a maximal subsystem of $X$ and this is impossible because $\overline{X \backslash \left(Z \cup \mathcal{E}(X)\right)}$ is not contained in $\mathcal{E}(X)$. 
    Since $\mathcal{E}(X)$ is closed, we have $\overline{\mathcal{K}(X)} \backslash Z \subset \mathcal{E}(X)$. 
    This implies that $\mathcal{E}(\mathcal{K}(X)) \subset \mathcal{E}(X)$.
    Therefore, since $X = \mathcal{K}(X) \cup \mathcal{E}(X) = \mathcal{K}^2(X) \cup \mathcal{E}(\mathcal{K}(X)) \cup \mathcal{E}(X)$, $X = \mathcal{K}^2(X)\cup \mathcal{E}(X)$. By Lemma~\ref{lemma.nucleus}, $\mathcal{K}^2(X)$ contains $\mathcal{K}(X)$. However this is impossible because $\mathcal{K}^2(X)$ is contained in every maximal subsystem of $\mathcal{K}(X)$. Assume that  $\mathcal{K}(X)$ is minimal. By Proposition~\ref{proposition.minimal}, $\mathcal{E}(X)$ is a maximal subsystem of $X$. By definition, $\mathcal{K}(X)$ is thus included in $\mathcal{E}(X)$. This implies that $\mathcal{E}(X) = X$ and that $\mathcal{K}(X)$ is empty.
\end{proof}

% \comm{A}{c'est quoi $K$ dans la prueve du lemme 3.6 ? }

% \comm{(S)}{J'ai changé la preuve.}

\begin{lemma}
For every shift $X$ and $T \in \mathcal{T}(X)$, $T$ is minimal or transitive with maximality type $1$.
\end{lemma}

\begin{proof}
    This is a direct consequence of definitions and Lemma~\ref{lemma:maxtype}.
\end{proof}

\begin{lemma}
Consider a shift $X$. For all $T,T' \in \mathcal{T}(X)$, if $T \subset T'$ then $T=T'$. Furthermore, for all $T \in \mathcal{T}(X)$, $T \not \subset \mathcal{K}(X)$. Furthermore, if $\mathcal{T}(X)$ is infinite, and $(T_k)_{k \ge 1}$ is an enumeration of its elements (recall that $\mathcal{M}(X)$, and thus $\mathcal{T}(X)$, is countable [Corollary \ref{corollary.countable}]), then for all $m \ge 1$ we have \[T_m \not \subset \bigcap_k \overline{\bigcup_{l \ge k} T_{l}}.\]
\end{lemma}

\begin{proof}
\begin{enumerate}
\item Consider $T$ and $T'$ two elements of $\mathcal{T}(X)$ such that $T \subset T'$. Denote by $Z,Z'$ maximal subsystems of $X$ such that $T= \overline{X\backslash Z}$ and $T' = \overline{X \backslash Z'}$. Assume ad absurdum that $T \neq T'$. This implies that $Z \neq Z'$. We know by  Lemma~\ref{lemma.non.inclusion} that $T'$ is contained in $Z$, and thus $T$ is contained in $Z$, which is impossible. 
\item Assume that for some $T \in \mathcal{T}(X)$, $T \subset \mathcal{K}(X)$. Consider a $Z$ maximal subsystem of $X$ such that $T = \overline{X\backslash Z}$.
% \footnote{\textcolor{red}{j'ai chang\'e $Z\backslash X$ par $X\backslash Z$}}. 
We know (Lemma~\ref{lemma.max.dec}) that $\mathcal{K}(X)$ is contained in all the maximal subsystems of $X$. Therefore $T$ is contained in $Z$. Again, this is impossible.
\item The last point is proved in a similar way, as the set $\bigcap_k \overline{\bigcup_{l \ge k} T_{l}}$ is contained in every maximal subsystem of $X$.
\end{enumerate}
\end{proof}

% \comm{A}{$\mathcal{T}(X)$ countable? 3.8 and 3.9}

\subsection{\label{section.other.dec}Uniqueness}

\begin{theorem}\label{thm:decomposition}
    Consider a shift $X$, $K$ a set, 
    and $\mathcal{T}$ a set of shifts such that: \textbf{(i)} every element
    of $\mathcal{T}$ is minimal or transitive with maximality type $1$; 
    \textbf{(ii)} $K$ is empty or a shift of maximality type 0 and not minimal; 
    \textbf{(iii)} 
    $X = K \cup \overline{\bigcup_{T\in \mathcal{T}} T}$; \textbf{(iv)} if $T \subset T'$ for $T,T' \in \mathcal{T}$, then $T=T'$; \textbf{(v)}
        for all $T \in \mathcal{T}$, $T \not \subset K$. \textbf{(vi)} if $|\mathcal{T}|= \infty$, denoting by $T_1,T_2, \ldots$
        the elements of $\mathcal{T}$, for all $m$ we have 
        \[T_m \not \subset \bigcap_k \overline{\bigcup_{l \ge k} T_{l}}.\]
        Then the maximal subsystems of $X$ are subsystems of the form \[K \bigcup \overline{\left(\bigcup_{\underset{T' \in \mathcal{T}}{T' \neq T}} T'\right)} \bigcup M(T),\] 
        where $T \in \mathcal{T}$. Furthermore, $\mathcal{T}(X) = \mathcal{T}$ and $\mathcal{K}(X) \subset K$.
\end{theorem}

For a minimal shift $T$, we use the convention $M(T) = \emptyset$.

\begin{proof}
    For every $T \in \mathcal{T}$, denote by $S(T)$ the set \[K \bigcup \overline{\left(\bigcup_{\underset{T' \in \mathcal{T}}{T' \neq T}} T'\right)}\bigcup M(T).\]
The set $X \backslash S(T)$ contains $R(T)$ because of properties \textbf{(iv)}, \textbf{(v)} and \textbf{(vi)}. Because of the property \textbf{(iii)}, $X \backslash S(T)$ is contained in $T$. Since it does not contain any element of $M(T)$, $X \backslash S(T) = R(T)$. This implies that $S(T)$ is maximal.
Consider some strict subsystem $Z$ of $X$. If there is some $T \in \mathcal{T}$ which is not contained in $Z$, then $Z \subset S(T)$. If $Z$ contains every $T \in \mathcal{T}$, then $K \neq \emptyset$. Thus $Z \cap K$ is a strict subsystem of $K$ or is empty, and because of property \textbf{(ii)}, $Z$ is not maximal. We have proved that the maximal subsystems of $X$ are the shifts $S(T)$, $T \in \mathcal{T}$. Along the way we have proved that for all $T \in \mathcal{T}$, $X \backslash S(T) = R(T)$, and thus $\overline{X \backslash S(T)} = T$. This implies that $\mathcal{T}(X) = \mathcal{T}$. From Lemma~\ref{lemma.nucleus} we have that $\mathcal{K}(X) \subset K$.
\end{proof}

% \comm{A}{see comment 55}

\section{Isolated points in the space of transitive shifts\label{section.transitive}}

M. Hochman, and R. Pavlov and S. Schmieding have studied the effects of transitivity on genericity. It is natural to wonder how our characterization in Section~\ref{section.isolated.shifts}
is affected by transitivity, and more generally by transitivity-like properties.
For this we will consider how 
these kind of properties affect the 
structure of the set of subsystems.

We denote by $\mathcal{T}^d$ the space of transitive $d$-dimensional shifts.

%\begin{remark}
%    If a SFT is transitive, then it has at most one maximal subsystem. Are there constraints on the maximal subsystem when there is one ?
%\end{remark}

\subsection{{The number of maximal subsystems of a transitive shift}}

\begin{lemma}
    A shift $X$ is transitive if and only if 
    \[X \neq \bigcup_{Z \in \mathcal{S}(X)} Z.\]
\end{lemma}

\begin{proof}
$(\Rightarrow)$ If $X$ is transitive, then there exists $x \in X$ such that $\overline{\mathcal{O}(x)} = X$. As a consequence there is no $Z \in \mathcal{S}(X)$ such that $x \in Z$, 
thus $x$ does not belong to the union of the elements of $\mathcal{S}(X)$.
$(\Leftarrow)$ Reciprocally, if the union of the elements of $\mathcal{S}(X)$ is different from $X$, then consider $x$ in $X$ outside of this union. The set $\overline{\mathcal{O}(x)}$ cannot be an element of $\mathcal{S}(X)$, thus it must be equal to $X$. This means that $X$ is transitive.
\end{proof}

\begin{theorem}\label{theorem.0.1}
    The maximality type of a transitive shift is either $0$ or $1$.
\end{theorem}

\begin{proof}
Let $X$ be a transitive shift. If $X$ is minimal, then its maximality type is 0. From now on, we assume that $X$ is not minimal. Consider the shift $Z \subset X$ which is equal to the closure of the union of all the strict subsystems of $X$. If $Z$ is different from $X$, then $Z$ is the only maximal subsystem of $X$, for every other subsystem is included in it. When $Z$ is equal to $X$, no strict subsystem $Z'$ can be maximal. Indeed, assume there is such a $Z'$. Since there exists some $x \in X$ whose orbit is dense in $X$, for each $Z''$ strict subsystem of $X$, $Z' \cup Z''$ is a strict subsystem and thus must be contained in $Z'$, which implies that $Z'' \subset Z'$. This means that $Z$ is not equal to $X$, which is a contradiction.
\end{proof}

\subsection{Characterization of isolated points in $\mathcal{T}^d$}

In this section, we provide a characterization of isolated points in the space of transitive shifts. It is straightforward that, since $\mathcal{T}^d \subset \mathcal{S}^d$, if a transitive shift is isolated in $\mathcal{S}^d$, it is also isolated in $\mathcal{T}^d$. We will see in Section~\ref{section.example.isolated.trans} that there are isolated shifts in $\mathcal{T}^d$ which are not isolated in $\mathcal{S}^d$.

We analyze separately shifts that have maximality type 0 (Section~\ref{section.max.type.0}) and the ones that have maximality type 1 (Section~\ref{section.max.type.1}). Both cases combined yield the claims in \cref{thm:B}.

\subsubsection{Maximality type 1\label{section.max.type.1}}

\begin{lemma}\label{lemma.type.1.star}
    If $X\in\mathcal{T}^d$ has maximality type 1, then every strict subsystem of $X$ is contained in its maximal subsystem.
\end{lemma}

\begin{proof}
    Denote by $Z$ the maximal subsystem. It is sufficient to prove that for all $x \in X \backslash Z$, $\overline{\mathcal{O}(x)} = X$. Let $x \in X \backslash Z$. We have then $Z \cup \overline{\mathcal{O}(x)} = X$, otherwise $Z$ would not be maximal. Since 
    $X$ is transitive, there exists $z \in X$ such that $\overline{\mathcal{O}(z)} = X$, and this implies that $z \in X \backslash Z$. As a consequence, $z \in \overline{\mathcal{O}(x)}$ and $\overline{\mathcal{O}(x)} = X$.
\end{proof}

\begin{comment}
Un type 1 transitif est isolé si et seulement si il est contenu dans un SFT dans lequel son complémentaire n'est pas dense.
-> c'est équivalent à ce que le complémentaire d'une orbite dense est un maximal.
\end{comment}

\cref{lemma.type.1.star} and \cref{thm.main} have as a consequence the following.

\begin{corollary}
A transitive shift of finite type $X$ which has maximality type 1 is isolated in $\mathcal{S}^d$ and in $\mathcal{T}^d$. %Furthermore, any shift of maximality type $1$ is transitive.
\end{corollary}

\begin{remark}
Observe that maximality type 0 does not imply transitivity, because a union of full shifts on disjoint alphabets $\mathcal{A}_1,\mathcal{A}_2$
such that $|\mathcal{A}_1| > 1$ or $|\mathcal{A}_2| > 1$ is of maximality type 0 but it is not transitive. However, is there a shift of maximality type 0 which is not a union of transitive shifts? Similarly, a maximality type $1$ shift is not necessarily transitive, because the disjoint union of a maximality type $0$ shift and a maximality type $1$ shift has maximality type $1$, and can not be transitive.
\end{remark}

\begin{example}
We provide an example of a transitive shift with maximality type $1$.
Let $\textbf{u}=(\textbf{u}_1,\textbf{u}_2)\in\mathbb{Z}^2$ and let $\overline{\mathcal{O}(x)}$ be a bi-dimensional shift, 
where $x$ is the configuration such that $x_{\textbf{u}}$ is equal to: %\textcolor{blue}{all of the following conditions are in terms of $u_2$, is that right? Also, the goal of the example does not seem too clear, like the conclusion of  it}
$\uparrow$ if $\textbf{u}_2 > 0$; $\downarrow$ if $\textbf{u}_2 < 0$; $\rightarrow$ if $\textbf{u}_2 = 0$ and $\textbf{u}_1 > 0$; $\leftarrow$ if $\textbf{u}_2 = 0$ and $\textbf{u}_1 < 0$; 
otherwise it is equal to $\square$. This shift is of finite type, transitive, and has a unique maximal subsystem, which is the set of configurations in which the symbol $\square$ does not appear.
\end{example}

%\begin{lemma}\label{lemma.approx.subsystem}
%    Consider a shift of finite type $Z$ and a subsystem $X$ of $Z$. If a sequence $X_n$ of shifts converges towards $X$, then there exists some $n_0$ such that for all $n \ge n_0$, $X_n \subset Z$.
%\end{lemma}

%\begin{proof}
%Denote by $\mathcal{A}$ the alphabet of $Z$. Since $Z$ is a shift of finite type, there exists a finite subset $\mathbb{U}$ of $\mathbb{Z}^d$ such that $Z = X_{\mathcal{A}^{\mathbb{U}} \backslash \mathscr{L}_{\mathbb{U}}(Z)}$. Since $X \subset Z$, we have $\mathscr{L}_{\mathbb{U}}(X) \subset \mathscr{L}_{\mathbb{U}}(Z)$ and since $X_n \rightarrow X$, there exists $n_0$ such that for all $n \ge n_0$, $\mathscr{L}_{\mathbb{U}}(X_n) = \mathscr{L}_{\mathbb{U}}(X)$. This means that $X_n \subset Z$.
%\end{proof}

\begin{proposition}\label{proposition.transitive.1}
    A transitive shift $X$ of maximality type 1 which is not of finite type is isolated in $\mathcal{T}^d$
    if and only if there exists a shift of finite type $Z$ 
    such that $X \subset Z$ and $Z \backslash X$ is not dense in $Z$.
\end{proposition}

\begin{proof}
    $(\Leftarrow)$ In this case we have that $\overline{Z\setminus X}$ does not intersect $R(X)$, 
    % \textcolor{blue}{, otherwise $\overline{Z\backslash X}\supset X$, as the orbit closure of any point in $R(X)$ must be $X$. This means that $Z=\overline{Z\backslash X}$, since $X\subset Z$, which is impossible since by hypothesis $Z\backslash X$ is not dense in $Z$.} 
    otherwise, we would have that $ X \subset \overline{Z\setminus X}$, since $Z\setminus X$ is invariant under $\sigma$. Using the trivial inclusion $Z \setminus X \subset \overline{Z\setminus X}$ and taking the union with $X\subset\overline{Z\backslash X}$, we get $Z = X \cup (X\backslash Z) \subset \overline{Z\backslash X}$. Since $\overline{Z\backslash X} \subset Z$, we get that $\overline{Z\setminus X}=Z$, which contradicts the hypothesis that $Z\setminus X$ is not dense in $Z$.
    Let us assume ad absurdum that there is a sequence $(x_n)_{n\ge0}$ such that $\overline{\mathcal{O}(x_n)} \rightarrow X$ and for all $n$, $\overline{\mathcal{O}(x_n)} \neq X$. Since $Z$ is of finite type, by application of Lemma~\ref{lemma.subsystem.2}, we can assume that for all $n$, $\overline{\mathcal{O}(x_n)} \subset Z$. 
    Denote by $X'$ the unique maximal subsystem of $X$. Since $\overline{\mathcal{O}(x_n)} \neq X$, we have that either $X_n := \overline{\mathcal{O}(x_n)} \subset \overline{Z \backslash X} \cup X'$ or $X_n$ strictly contains $X$. If there were infinitely many integers $n$ such that $X_n$ strictly contains $X$, $\overline{Z\backslash X}$ would then have non empty intersection with $R(X)$, which is not true. Therefore we can assume that for all $n$, $\overline{\mathcal{O}(x_n)} \subset \overline{Z \backslash X} \cup X'$. Hence, by taking the limit, $X \subset \overline{Z \backslash X} \cup X'$, which is impossible. $(\Rightarrow)$ Suppose now that for all $Z$ shift of finite type which contains $X$, $Z \backslash X$ is dense in $Z$. Denote by $(Z_n)_{n\ge0}$ a non-increasing sequence of shifts of finite type which converges towards $X$. For all $n$, there exists a sequence of points $(x_{m,n})_{m\ge0}$ in $Z_n \backslash X$ which converges to some point in $R(X)$. Observe that $\mathcal{S}(Z_1)$ is compact, thus we can assume, without loss of generality, that $\overline{\mathcal{O}(x_{m,n})}$ converges when $m \rightarrow +\infty$. The limit, that we denote by $\Lambda_n$, must contain $X$ and is contained in $Z_n$. We can assume without loss of generality, again because $\mathcal{S}(Z_1)$ is compact, that $(\Lambda_n)_n$ converges. Since $X \subset \Lambda_n \subset Z_n$ for all $n$ and $Z_n \rightarrow X$, $\Lambda_n$ converges towards $X$. This implies that there is a sequence of points $(z_n)_{n\ge0}$ such that $\overline{\mathcal{O}(z_n)} \rightarrow X$. This means that $X$ is not isolated in $\mathcal{T}^d$.
\end{proof}

\begin{example}
    An example of (non finite type) transitive shift of maximality type $1$, and not isolated in $\mathcal{T}^d$ is the \textit{sunny-side-up}, which has alphabet $\{0,1\}$ and whose configurations are the ones which contain at most one symbol $1$. Indeed, it has a unique subsystem, which is thus maximal, whose only element is the constant configuration with value $0$. This shift is the limit of the orbit closures of configurations in which $1$ symbols appear periodically in both directions with period $n$.
\end{example}

%\begin{example}
%All the transitive subsystems of the Robinson shift, except for its unique minimal subsystem, are examples of non finite type transitive shifts which isolated in $\mathcal{T}_d$, and because they are not of finite type, they are not isolated in $\mathcal{S}_d$. Indeed, let us denote by $P_n$ any square pattern of the Robinson subshift which contains a desynchronisation at level $n$ and which is of minimal size denoted by $k_n$. Let us denote by $X_n$ the subsystem of the Robinson subshift corresponding to forbidden patterns $P_m$, $m \neq n$. Then the subshift $X_n$ is transitive for all $n$. Furthermore, the distance between $X_n$ and any other $X_m$ is at least $2^{-k_n}$. Since there is a neighborhood of $X_n$ which contains only subsystems of the Robinson subshift, there is a neighborhood of $X_n$ which contains only such subsystems which are different from the $X_m$, $m \neq n$.
%\end{example}

\subsubsection{{Maximality type 0}\label{section.max.type.0}}

\begin{lemma}\label{lemma.converging.minimal}
    Consider a minimal shift $X$ and a sequence of shifts $(X_n)_{n\ge0}$ such that $X_n \rightarrow X$. 
    For every sequence of shifts $(Z_n)_{n\ge0}$ such that for all $n$, $Z_n \subset X_n$, we have that $Z_n \rightarrow X$.
\end{lemma}

\begin{proof}
    It is sufficient to see that the limit of any convergent subsequence of $(Z_n)$ has to be a subsystem of $X$, therefore it is equal to $X$.
\end{proof}

\begin{proposition}
Let $X$ be a transitive shift which has maximality type 0. %If $X$ is not minimal, then it is not isolated in $\mathcal{T}^d$. If $X$ is minimal, then it is isolated in $\mathcal{T}^d$ if and only if it is of finite type.
Then $X$ is isolated in $\mathcal{T}^d$ if and only if it is a minimal shift of finite type.
\end{proposition}

\begin{proof}
    Consider a transitive shift $X$ which has maximality type 0.
    \begin{enumerate}
        \item \textbf{If $X$ is minimal.} \begin{enumerate}
            \item \textbf{If X is not of finite type.} Consider a set $\mathcal{F}$ such that $X = X_\mathcal{F}$, and a sequence of patterns $(p_n)_{n \ge 0}$ such that $\mathcal{F} = \{p_0, p_1 , \ldots\}$. Then set $X_n = X_{\mathcal{F}_n}$, where $\mathcal{F}_n = \{p_0 , \ldots , p_n\}$. Then we have $X_n \rightarrow X$ and for all $n$, $X_{n+1} \subset X_n$ and $X_n \neq X$.  We have then that for all $n$, there exists a configuration $x_n \in X_n$ with $x_n \notin X$. Thus, for all $n$, $\overline{\mathcal{O}(x_n)} \neq X$ and $\overline{\mathcal{O}(x_n)} \rightarrow X$, by an application of Lemma~\ref{lemma.converging.minimal}. %Since every convergent subsequence converges towards $X$, then $\overline{\mathcal{O}(x_n)} \rightarrow X$. 
            We conclude that $X$ is not isolated in $\mathcal{T}^d$.
        \item \textbf{If X is of finite type.} In this case, since it has no strict subsystem, by \cref{thm.main}, it is isolated in $\mathcal{S}^d$ and thus in $\mathcal{T}^d$.
        \end{enumerate}
        \item \textbf{If $X$ is not minimal.} Let $x \in X$ such that $\overline{\mathcal{O}(x)} = X$. By hypothesis, the closure of the union of all elements of $\mathcal{S}(X)$ 
        is equal to $X$. Therefore, there exists a sequence of points $(x_n)_{n\ge0}$ such that 
        for all $n$,  $\overline{\mathcal{O}(x_n)}$ is a strict subsystem of $X$ and such that $x_n \rightarrow x$. We claim that $\overline{\mathcal{O}(x_n)} \rightarrow X$. Let us prove that for all $\epsilon > 0$, 
        there exists some $n$ such that the Hausdorff distance 
        between $\overline{\mathcal{O}(x_n)}$ and $X$ is smaller than $\epsilon$. Since $\mathcal{O}(x)$ is dense in $X$, there exists a finite subset $S$ of $\mathcal{O}(x)$
        such that $\delta_H(S,X) < \epsilon/2$. Since $x_n \rightarrow x$, by continuity, there exists  $n$ and a subset $S_n$ of $\overline{\mathcal{O}(x_n)}$ such that $\delta_H(S_n,S) < \epsilon/2$. Consequently $\delta_H(\overline{\mathcal{O}(x_n)},X) < \epsilon$.
    \end{enumerate}
\end{proof}

\begin{proposition}\label{proposition.annex.gluing}
    Let $X$ be a transitive, infinite shift. If the set of its periodic points is dense, then $X$ has maximality type 0.
\end{proposition}

\begin{proof}
    In this case, there exists a strictly increasing sequence of shifts different from $X$ which converges towards $X$. Therefore, these shifts are not included in any maximal subsystem. By Lemma~\ref{lemma.type.1.star}, $X$ cannot have maximality type 1. Since $X$ is transitive, Theorem~\ref{theorem.0.1} implies that it has maximality type 0.
\end{proof}

\begin{example}\label{example.maxtype.fullshift}
    By the means of Proposition~\ref{proposition.annex.gluing}, we have that 
    every full shift has maximality type $0$ and, thus, it is not isolated in $\mathcal{T}^d$.
\end{example}

\subsubsection{Example of a non finite type shift isolated in $\mathcal{T}^d$\label{section.example.isolated.trans}}\label{non SFT isolated in T}

We provide an example of a transitive shift of maximality type 1 which is isolated in $\mathcal{T}^2$ and is not of finite type, which implies that it is not isolated in $\mathcal{S}^2$. This example can be easily adapted in order to construct shifts isolated in $\mathcal{T}^d$, $d>2$, which are not of finite type. \bigskip

% Let us denote by $\mathcal{A}_{\textbf{t}}$ the 
 % following alphabet: 
 % \[\rightarrow,\leftarrow,\uparrow,\downarrow, \square,\]
 Denote by $\mathbb{N}^*$ the set of positive integers.
 We denote by $x$ the configuration such that 
$x_{(0,0)} = \square$; for all $k > 0$ (resp $k< 0$), $x_{(0,k)} = \uparrow$ (resp. $\downarrow$); for all $k >0$ (resp. $k < 0$), $x_{(k,0)} = \rightarrow$ (resp. $\leftarrow$); the restriction of $x$ to any set $\left( \pm \mathbb{N}^*\right) \times \left( \pm \mathbb{N}^*\right)$ is a pattern over an infinite supertile of the Robinson shift $X_R$. \bigskip

The shift $Z:= \overline{\mathcal{O}(x)} \bigcup X_R^{\infty}$ is a subsystem of $X:= \overline{\mathcal{O}(x)} \bigcup X_R$, which is of finite type. Furthermore, 
$Z$ is clearly a transitive shift whose maximality type is $1$. Moreover, $X \backslash Z$ is not dense in $X$, because no configuration in it contains the symbol $\square$. The shift $Z$ is not of finite type because 
it is the limit of a decreasing sequence $X_n$ of shifts of finite type, 
where $X_n = \overline{\mathcal{O}(x)} \bigcup X_R^n$.
By Proposition~\ref{proposition.transitive.1}, $Z$ is a transitive shift which has maximality type $1$, and it is isolated and not of finite type.

%\begin{proposition}
    % There exists a homeomorphism $\Phi$ from the space of transitive subsystems of $X_R$ to the space
    % \[\bigcup_{n \in \mathbb{N}^*}\left\{\frac{1}{n}\right\} \times  \left( \frac{1}{n} \textbf{c}\right)\]
    % such that $\Phi(X_R^{\infty}) = 0$.
% \end{proposition}

% \begin{proof}
    % Let us consider a configuration of $X_R \backslash X_R^{\infty}$ which satisfies condition (ii). Since it is not in $X_R^{\infty}$, there is a minimal integer such that the order $n$ supertiles in one infinite supertile are not aligned with the order $n$ supertiles of the other. We denote then $X_R^m$ the set of configurations such that that this integer is $m$.
% \end{proof}

% \begin{corollary}
    % All the transitive non minimal subsystems of the Robinson are isolated in $\mathcal{T}_d$.
% \end{corollary}

\subsection{{Other dynamical properties}}

Denote by $\mathcal{M}^d$ the space of minimal $d$-dimensional shifts.

\begin{proposition}
    A minimal shift $X$ is isolated in $\mathcal{M}^d$ if and only if it has a neighborhood in $\mathcal{S}^d$ in which every shift contains $X$.
\end{proposition}

\begin{proof}
    $(\Leftarrow)$ Suppose that there is a sequence of shifts $(X_n)_{n\ge0}$ in $\mathcal{S}^d$ which converges towards $X$ and such that $X_n$ does not contain $X$, for all $n$. This means that $X_n$ contains another a minimal shift $Z_n \neq X$. By Lemma~\ref{lemma.converging.minimal}, $Z_n \rightarrow X$, and therefore $X$ is not isolated in $\mathcal{M}^d$. $(\Rightarrow)$ Reciprocally, if $X$ is not isolated in $\mathcal{M}^d$, there is a sequence of different minimal shifts 
    which converges to $X$, and since they are minimal, none of these shifts can contain $X$.
\end{proof}

Recall that we denote by $||.||_{\infty}$ the norm on 
$\mathbb{Z}^d$ such that $||\textbf{u}||_{\infty} = \max_{k \le d} |\textbf{u}_k|$, for all $\textbf{u} \in \mathbb{Z}^d$. 
% We say that a shift $X$ is \textit{mixing} when for all pattern $p,q$ in the language of $X$, denoting by $\mathbb{U},\mathbb{V}$ supports of $p,q$ respectively, there exists some $m$ such that for all $\textbf{u} \in \mathbb{Z}$ such that $||\textbf{u}||_{\infty} \ge m$, there is a configuration of $X$ whose restriction to $\mathbb{U}$ is $p$ and whose restriction to $\textbf{u}+\mathbb{V}$ is equal to $q$.
The Hausdorff distance between two finite subsets $\mathbb{U}$ and $\mathbb{V}$ of $\mathbb{Z}^d$ is 
\[\max \left( \max_{\textbf{u} \in \mathbb{U}} \min_{\textbf{v} \in \mathbb{V}} ||\textbf{u}-\textbf{v}||_{\infty} , \max_{\textbf{v} \in \mathbb{V}} \min_{\textbf{u} \in \mathbb{U}} ||\textbf{u}-\textbf{v}||_{\infty}\right)\]

% \begin{proposition}
    % All mixing shifts have maximality type $0$.
% \end{proposition}

% \begin{proof}
    % Let us consider $X$ a mixing shift and $\mathcal{F}$ a set of patterns such that $X=X_{\mathcal{F}}$. If $X$ is minimal, it has maximality type $0$. 
    % For the remainder of the proof let us assume that it is not minimal. In particular its language has at least two elements $p,q$ which have the same support. We know that
    % if $X$ has a maximal subsystem, it is equal to $X_{\mathcal{F} \cup \{w\}}$ for some pattern $w$ not in $\mathcal{F}$. In order to prove that $X$ has no maximal subsystem, 
    % we can prove that for all $w$ not in $\mathcal{F}$, there exists $w'$ such that $X_{\mathcal{F} \cup \{w'\}}$ strictly contains $X_{\mathcal{F} \cup \{w\}}$. Let us fix such pattern $w$ and denote by $\mathbb{U}$ a support of $w$. Denote by $\mathbb{V}$ a (common) support of $p,q$. Because $X$ is mixing, there is some $\textbf{u}$ and a configuration $x \in X$ 
    % (resp $x' \in X$) such that the restriction of $x$ on $\mathbb{U}$  is $w$ and the restriction
    % of $x$ to $\textbf{u}+\mathbb{V}$ is $p$ (resp. $q$). Let us denote then by $w'$ (resp. $w''$) the pattern on support $(\textbf{u}+\mathbb{V})\cup \mathbb{U}$ whose restriction on $\textbf{u}+\mathbb{V}$ is $p$ (resp. $q$) and whose restriction on $\mathbb{U}$ is $w$.
    % $X_{\mathcal{F} \cup \{w'\}}$ strictly contains $X_{\mathcal{F} \cup \{w\}}$ because the language of the former contains the pattern $w$ (because it contains $w''$) and not the former.
% \end{proof}

\begin{definition}
Let $f: \mathbb{N} \rightarrow \mathbb{N}$ be a map.
   We say that a shift $X$ is $f(n)$-block gluing when for all $n$, for all patterns $p,q$ with support $\mathbb{B}_n^d$, and for all $\textbf{u} \in \mathbb{Z}^d$ such that the Hausdorff distance between 
   $\mathbb{B}_n^d$ and $\textbf{u} + \mathbb{B}_n^d$ is at least $n+f(n)$, there exists a configuration $x \in X$ such that $x_{\mathbb{B}_n^d} = p$ and $x_{\textbf{u} + \mathbb{B}_n^d} = q$.
\end{definition}

\begin{proposition}
    An $f(n)$-block gluing shift, where $f(n) = o(\log(n))$, has maximality type 0. 
\end{proposition}

\begin{proof}
    It is known that a shift which is $o(\log(n))$-block gluing has a dense set of periodic points (see for instance~\cite{GangloffSablik}). Furthermore, for every $f$, $f(n)$-block gluing implies transitivity. If the shift is finite, then it is of maximality type 0. If it is infinite, then Proposition~\ref{proposition.annex.gluing} applies, and implies that the shift is of maximality type 0.
\end{proof}

\begin{question}
    What are the possible maximality types for shifts which are strictly $f(n)$-block gluing, where $f$ is linear? For mixing shifts?
\end{question}

\section{The topological structure of $\mathcal{S}^d$ when $d \ge 2$\label{section.topological.structure}}

A natural approach to genericity for higher-dimensional shifts is to 
compute the closure of the set of isolated shifts, and in particular its complement 
in $\mathcal{S}^d$, especially because isolated points form a generic set contained in all other generic set when $d=1$ (see Section~\ref{section.closure.isolated}). As this question is quite difficult, we only provide some observations in this direction. Another way is to compute the Cantor-Bendixson rank of $\mathcal{S}^d$.
We prove in Section~\ref{section.cantor.bendixson} that this rank is infinite. We then prove in Section~\ref{section.cb.residue} that the closure of the set of maximality type $\infty$ shifts is included in the Cantor-Bendixson residue. 

\subsection{On the closure of the set of isolated shifts\label{section.closure.isolated}}

Denote by $\mathcal{I}^d$ the set of isolated shifts in $\mathcal{S}^d$. The question of interest is the following:

\begin{question}
    How to characterize the set $\overline{\mathcal{I}^d}$?
\end{question}

The answer is not trivial: $\overline{\mathcal{I}^d}$ 
is not equal to $\mathcal{I}^d$ and not equal to $\mathcal{S}^d$ (this was already proved in~\cite{Doucha}, \textbf{Theorem 3.1}), because of the following examples.

\begin{example}
    Any infinite full shift is in $\overline{\mathcal{I}^d}$ because it is the limit of a sequence of finite shifts and a finite shift is isolated. On the other hand, a full shift has no maximal subsystem and has strict subsystems, and therefore is not isolated in $\mathcal{S}^d$.
\end{example}

\begin{example}
    %The shift $X$ presented in Example~\ref{example.inf.max.sub} has only finitely many subsystems which are isolated in $\mathcal{S}^d$. Indeed, a subsystem is of finite type if it is obtained from $X$ by forbidding 
    %the cross symbols to be spaced by a finite number of distances, or by forbidding the cross symbol and additional patterns. \textcolor{red}{reasoning problem here} There are finitely many subsystems of the second type. The other ones are of maximality type $+\infty$, for the same reason as $X$ has infinitely many maximal subsystems. Thus they are not isolated. As a consequence, $X$ is not in $\overline{\mathcal{I}_d}$. 
    The shift $X_R$ has no isolated subsystem, and is thus not in $\overline{\mathcal{I}_d}$. We leave the proof to the reader.
\end{example}

\begin{example}
    Every transitive subsystem $Z$ of $X_R$ is not isolated in 
    $\mathcal{T}^d$. The reason is that for any shift of finite type $X$ which contains $Z$, $X\backslash Z$ is dense in $X$ (Proposition~\ref{proposition.transitive.1} applies). Hence the set of isolated points of $\mathcal{T}^d$ is not dense in $\mathcal{T}^d$. 
\end{example}

\subsection{On the Cantor-Bendixson rank of $\mathcal{S}^d$\label{section.cantor.bendixson}}

\subsubsection{{Definition}}

We recall the definition of the Cantor-Bendixson rank. For every topological space $X$, denote by $X'$ the set of non-isolated points in $X$, called the \textbf{derived set} of $X$.

\begin{definition}
    For every ordinal number $\alpha$, the $\alpha$-th Cantor-Bendixson derivative of $X$ is the set $X^{(\alpha)}$ defined as follows: $X^{(0)} = X$; for all $\alpha$, $X^{(\alpha+1)} = {X^{(\alpha)}}'$; $X^{(\lambda)} = \bigcap_{\alpha < \lambda} X^{(\alpha)}$ where $\lambda$ is a limit ordinal. 
    The \textbf{Cantor-Bendixson rank} of $X$ is the least ordinal $\alpha$ such that $X^{(\alpha +1)} = X^{(\alpha)}$. We call \textbf{Cantor-Bendixson residue} the set $X^{(\alpha)}$, where $\alpha$ is the Cantor-Bendixon rank of $X$.
\end{definition}

For all $d$ and $n$, we denote by $\mathcal{S}_n^d$ the set $\left(\mathcal{S}^d\right)^{(n)}$.

\subsubsection{{The shift associated with the${}\times2\times3$ system on the circle}}

Using a symbolic encoding of the well-known${}\times 2 \times 3$  system on the circle, we first prove that there is at least one isolated point in ${\mathcal{S}^d}' = \mathcal{S}^d \backslash \mathcal{I}^d$.

\begin{notation}
We denote by $X_0$ the
shift on alphabet $\mathbb{Z}/6\mathbb{Z}$ defined as the set of configurations $x$ in $\left(\mathbb{Z}/6\mathbb{Z}\right)^{\mathbb{Z}^2}$ such that for all $\textbf{u}\in\mathbb{Z}^2$, there exist $\epsilon_1 \in \{0,1\}$ and $\epsilon_2 \in \{0,1,2\}$ such that $x_{\textbf{u} + \textbf{e}_1} = 2 x_{\textbf{u}} + \epsilon_1$ and $x_{\textbf{u} + \textbf{e}_2} = 3 x_{\textbf{u}} + \epsilon_2$.
\end{notation}

Note that this shift is identical to the one defined in \cite{rudolph}. The following can be checked case by case:

\begin{lemma}\label{lemma.det.diag}
For all $k,l \in \mathbb{Z}/6\mathbb{Z}$, there exists a unique pattern $q$ in $\mathscr{L}_{\llbracket 0 , 1 \rrbracket ^2}(X_0)$ such that $q_{(0,0)} = k$ and $q_{(1,1)}=l$.
\end{lemma}

The following is well-known:

\begin{theorem}[H. Furstenberg, \cite{Furstenberg}]\label{theorem.furstenberg}
    The only infinite subset of the circle $\mathbb{T}^1$, which is closed and invariant for the maps $x \mapsto 2x$ and $x \mapsto 3x$ is $\mathbb{T}^1$ itself.
\end{theorem}

% \textcolor{red}{Here we can address comments 61 and 63, about $\mathscr{L}_{\mathbb{N}^2}$ and its topology.}

% \comm{S}{We can just define this for infinite sets where it is defined first and just say it is the product of discrete topology that we put on this set}

% \comm{A}{unless i'm missing something, to define it for infinite sets it just suffices to delete the word ``finite.'' For now i used \st{finite} where they appeared in pag 5. I'm not so sure where to put de topology part, for the moment I added a line (in blue) just after the definition of $\mathscr{L}(X)$.}

\begin{notation}
    For all $n \in \mathbb{N}$, set $\textbf{u}^n :=(n,n)$. We denote by $\lambda$ the function 
    $\mathscr{L}_{\mathbb{N}^2}(X_0) \rightarrow \mathbb{T}^1$
    such that for all $p \in \mathscr{L}_{\mathbb{N}^2}(X_0)$,
    $\lambda(p)$ is the point $x \in \mathbb{T}^1$ whose representation in base $6$ is $(p_{\boldsymbol{u}^n})_n$.
    
    % Denote by $\lambda: \mathbb{T}^1 \rightarrow \mathcal{L}_{\mathbb{N}^2}(X_0)$ the map which to an element of $\mathbb{T}^1$ associates the pattern $p$ on $\mathbb{N}^2$ such that for all $\textbf{u}=(\textbf{u}_1,\textbf{u}_2) \in \mathbb{N}^2$, $p_{\textbf{u}} = \overline{k} \in \mathbb{Z}/6\mathbb{Z}$, $k \in \llbracket 0, 5 \llbracket$ if and only if $2 ^{\textbf{u}_1}  3^{\textbf{u}_2} \cdot x \in [k/6,(k+1)/6[$.
\end{notation}

\begin{lemma}
The map $\lambda$ is surjective and continuous. Furthermore, for all $x \in \mathbb{T}^1$, $\left|\lambda^{-1}(\{x\})\}\right| = 2$ if $x$ is of the form $m/6^l$, for some $m,l\in\N$, and 
$\left|\lambda^{-1}(\{x\})\right| = 1$ otherwise.
    % The map $\lambda$ is a bijection.
\end{lemma}

\begin{proof}
Continuity is straightforward. Surjectivity is obtained from repeated application of Lemma~\ref{lemma.det.diag} and the shift map. The rest of the statement comes from the definition of $\lambda$.
% We prove that $\lambda$ is injective, and then that it is surjective. We set $\textbf{u}^n :=(n,n)$.

% \begin{enumerate}
%     \item \textbf{Injectivity.}
%     In order to see that $\lambda$ is injective, it suffices to see that for all $x \in [0,1[$, the sequence $\left(\lambda(x)_{\textbf{u}^n}\right)_n$ is the decomposition in base $6$ of $x$, and that it uniquely determines $x$. 
%     \item \textbf{Surjectivity.} Let $p$ a pattern in $\mathcal{L}_{\mathbb{N}^2}(X_0)$. Observe that for all $k,l\in\mathbb{Z}/6\mathbb{Z}$ there exists a unique $m$ such that $2m=l$ and $3k=m$ (resp. $3m=l$ and $2k=m$). Therefore, the sequence $(p_{\textbf{u}^n})_{n\ge0}$ uniquely determines $p$. Let $x\in[0,1[$ be such that its base 6 decomposition is $p_{\textbf{u}^0}p_{\textbf{u}^1}\ldots p_{\textbf{u}^n}\ldots$. Since $p$ and $\lambda(x)$ coincide on the diagonal $\{\textbf{u}^n\mid n\in\mathbb{N}\}$, we have $p=\lambda(x)$.
%     \end{enumerate}
\end{proof}

\begin{lemma}\label{lemma.quarter.to.full}
For a bi-dimensional shift $X$, if $\mathscr{L}_{\mathbb{N}^2}(X)$ is finite, then $X$ is finite. Furthermore, considering two shifts $X \subset Z$, if  $\mathscr{L}_{\mathbb{N}^2}(X) = \mathscr{L}_{\mathbb{N}^2}(Z)$, then $Z=X$.
\end{lemma}

\begin{proof}
Suppose that $\mathscr{L}_{\mathbb{N}^2}(X)$ is finite.
For every $p \in \mathscr{L}_{\mathbb{N}^2}(X)$, $\left(p_{|(n,n) + \mathbb{N}^2}\right)_n$ is an ultimately periodic sequence. As a consequence the sequence $\left(\sigma^{(n,n)}(X)\right)_n$ converges towards a finite shift. On the other hand, it is constant with value $X$, therefore $X$ is finite.
Consider now $Z \subset X$ such that $\mathscr{L}_{\mathbb{N}^2}(X) = \mathscr{L}_{\mathbb{N}^2}(Z)$. In order to prove that $Z=X$, it suffices to prove that for all $x \in X$, and all $n$, there exists a configuration $z$ in $Z$ such that $z$ coincides with $x$ on $(-n,-n) + \mathbb{N}^2$. Indeed, in this case, $x$ is in the closure of $Z$, thus in $Z$. Let $x \in X$ and $n\in\mathbb{N}$. By hypothesis there exists $z \in Z$ such that $\sigma^{(n,n)}(x)$ and $z$ coincide on $\mathbb{N}^2$. Therefore, $x$ and  $\sigma^{(-n,-n)}(z) \in Z$ coincide on $(-n,-n) + \mathbb{N}^2$. This ends the proof.
\end{proof}

\begin{proposition}
All the strict subsystems of $X_0$ are finite.
\end{proposition}

\begin{proof}
Consider a subsystem $X$ of $X_0$.
    Denote by $Z$ the 
    set of restrictions of elements of $X$ to $\mathbb{N}^2$. Since $\lambda$ is continuous, and that by definition of $X_0$, for all $p \in \mathscr{L}_{\mathbb{N}^2}(X_0)$, $\lambda(\sigma^{\textbf{e}_1}(x)) = 2\lambda(p) $, and $\lambda(\sigma^{\textbf{e}_2}(p)) = 3\lambda(p) $, the set $\lambda(Z)$ is invariant for the maps $x \mapsto 2x$ and $x \mapsto 3x$. By Theorem~\ref{theorem.furstenberg},
    there are two possibilities.
    If the set $\lambda(Z)$ is finite, since $\left|\lambda^{-1}(\{x\})\right| \le 2$ for all $x \in \mathbb{T}^1$, then $Z$ is finite and by Lemma~\ref{lemma.quarter.to.full}, $X$ is also finite. Otherwise, as a consequence of Theorem~\ref{theorem.furstenberg}, $\lambda(Z)$ must be equal to $\mathbb{T}^1$ and since for all $x \in \mathbb{T}^1$ which is not of the form $m/6^l$, %$m/2^l$
    $\left|\lambda^{-1}(\{x\})\right| = 1$, $Z$ contains a subset which is dense in $\mathscr{L}_{\mathbb{N}^2}(X_0)$. Since it is closed, it must be equal to $\mathscr{L}_{\mathbb{N}^2}(X_0)$. Again, because of Lemma~\ref{lemma.quarter.to.full}, $X = X_0$. Therefore if $X$ is different from $X_0$, then $X$ is finite.
    % Consider a subsystem $X$ of $X_0$.
    % Denote by $Z$ the 
    % set of restrictions of elements of $X$ to $\mathbb{N}^2$. Since $\lambda^{-1}$ is continuous, and for all $p \in \mathcal{L}_{\mathbb{N}^2}(X_0)$, $\lambda^{-1}(\sigma^{\textbf{e}_1}(x)) = 2\lambda^{-1}(x) $, and $\lambda^{-1}(\sigma^{\textbf{e}_2}(x)) = 3\lambda^{-1}(x) $, the set $\lambda^{-1}(\mathcal{L}_{\mathbb{N}^2}(X_0))$ satisfies the conditions of Theorem~\ref{theorem.furstenberg}.
    % There are thus two possibilities.
    % If the set $\lambda^{-1}(Z)$ is finite, then $Z$ is finite and by Lemma~\ref{lemma.quarter.to.full}, $X$ is also finite. Otherwise, as a consequence of Theorem~\ref{theorem.furstenberg}, $\lambda^{-1}(Z)$ must be equal to $\mathbb{T}^1$ and because $\lambda$ is bijective, $Z = \mathcal{L}_{\mathbb{N}^2}(X_0)$. Again, because of Lemma~\ref{lemma.quarter.to.full}, $X = X_0$. Therefore if $X$ is different from $X_0$, then $X$ is finite.
\end{proof}

\begin{remark}
    As a consequence, the shift $X_0$ has maximality type 0.
\end{remark}

\begin{notation}
    For all $d \ge 0$, and $X$ a bi-dimensional shift,
    we denote by $\mathcal{E}^{d}(X)$ the $(d+2)$-dimensional shift such that for all $k \ge 3$, and for all $x \in \mathcal{E}^{d}(X)$, $\sigma^{\textbf{e}_k}(x) = x$ and, denoting by $P$ the set of elements of $\mathbb{Z}^d$ whose $k$-th coordinate is 0 for all $k \ge 3$, 
    $\mathscr{L}_P (\mathcal{E}^{d}(X)) = X_0$.
\end{notation}

\begin{theorem}\label{theorem.times.shift}
    For all $d \ge 0$, $\mathcal{E}^{d}(X_0)$ is isolated in $\mathcal{S}^{d+2} \backslash \mathcal{I}^{d+2}$.
\end{theorem}

\begin{proof}
    Since strict subsystems of $X_0$ are finite, it is straightforward that strict subsystems of $\mathcal{E}^{d}(X_0)$ are also finite for all $d \ge 0$.
    Since a finite shift is isolated, all the subsystems of $\mathcal{E}^{d}(X_0)$ are 
    isolated. Since $\mathcal{E}^{d}(X_0)$ is of finite type, it has a neighborhood which contains only isolated points in $\mathcal{S}^{d+2}$. 
    Therefore $\mathcal{E}^{d}(X_0)$ is isolated in $\mathcal{S}^{d+2} \backslash \mathcal{I}^{d+2}$.
\end{proof}

%X2X3:

%Question: comment caractériser la structure topologique de x2x3 et la reproduire ? En particulier produire des isolés dans S-I (éventuellement en construire avec deux mesures invariantes) -> adhérence de l'union d'une suite de périodiques aux languages croissant ? Clairement pas suffisant car le full shift est aussi limite d'une telle suite. -> ajouter que tout motif suffisamment grand d'un language contient tous les mots du précédent ? -> les configurations limites sont celles dans lesquelles tous les mots des sous shifts périodiques apparaissent.
%Pour qu'il n'y ai pas de sous système infini, il faut que toutes ces configurations soient denses. Autrement dit si une configuration contient tous les mots du language périodique, alors elle contient tous les mots du language. -> il faut une condition sur la densité dans un motif : pour tout motif dans le language fini, il existe un rang a partir duquel tous les motifs des périodiques contiennent une densité de ce motif supérieure à une constante.
%\end{comment}

\subsubsection{{The Cantor-Bendixson rank of the shifts space is infinite}}

In this subsection we prove \cref{thm:C} from the introduction as \cref{thm.Cantor}. For this purpose, we first prove some technical lemmas on which the proof of \cref{thm.Cantor} relies, making this proof simpler.

% Let us denote by $\mathcal{P}_{n}^d$ the sets defined by $\mathcal{P}_{0}^d = \mathcal{I}^d$ for all $n$, $\mathcal{P}_{n+1}^{d}$ is the set of shifts of finite type in $\mathcal{S}^d$ whose strict subsystems are all in $ \bigcup_{k \le n} \mathcal{P}_{k}^{d}$. 
% We also denote $\mathcal{D}^d_{n}:=\mathcal{P}_{n}^{d} \backslash \mathcal{P}_{n-1}^{d}$ when $n \ge 1$ and $\mathcal{D}^d_{0} := \mathcal{P}_{0}^{d}$. The following is straightforward:

% \begin{lemma}\label{lemma.cantor.inclusion}
%     For all $d,n$, $\mathcal{D}_{n}^{d} \subset \mathcal{S}_{n}^d$. 
%     %\textcolor{blue}{$S^d_n$?? the n-th derivative of $S^d$?? in that case ${S^d}^{(n)}$ but we should add the 'recursive' notation after the definition of the derived set $X'$}
% \end{lemma}

\begin{lemma}\label{lemma.union.2}
    The set $\mathcal{I}^d$ is invariant under disjoint unions.
\end{lemma}

\begin{proof}
    Indeed, the disjoint union of two shifts of finite type is also of finite type.
    Furthermore, if $X$ and $Z$ are two disjoint isolated shifts, the maximal subsystems 
    of $X \cup Z$ are of the form $X \cup M$
    or $N \cup Z$, where $N$ and $M$ are maximal subsystems of $X$ and $Z$, respectively. Therefore $X \cup Z$ has a finite number of maximal subsystems.
    Since every subsystem of $X \cup Z$ is the union of a subsystem of $X$ and a subsystem of $Z$, 
    every subsystem of $X \cup Z$ is included in one of its maximal subsystems. The result follows from \cref{thm.main}.
\end{proof}

For all $n > 0$, we denote by $\mathcal{D}_{n}^d$ the set of isolated points in $\mathcal{S}_n^d$ whose subsystems are all shifts of finite type. Also set $\mathcal{C}_n^d := \bigcup_{k=0}^n \mathcal{D}_k^d$. In particular, $\mathcal{D}_0^d = \mathcal{C}_0^d$ is the set of shifts in $\mathcal{I}^d$ whose subsystems are all of finite type. Thus every finite shift is in this set. Denote as well the lexicographic order on $\mathbb{N}^2$ by $<_{\texttt{lex}}$, which is defined as follows: for $(k,l),(m,n) \in \mathbb{N}^2$, $(k,l) <_{\texttt{lex}} (m,n)$ if and only if ($k < m$ or ($k=m$ and $l < n$)).

\begin{lemma}\label{lemma.recur.s}
    For all $n \ge 1$ and $0 \le m \le n$, we have the following property denoted by $P_{n,m}$. For every $X \in \mathcal{S}_m^d$ and $Z \in \mathcal{S}_n^d$ such that $X \cap Z = \emptyset$, if $m=0$ we have $X \cup Z \in \mathcal{S}_n^d$, otherwise we have $X \cup Z \in \mathcal{S}_{n+1}^d$.
\end{lemma}

\begin{proof}
    We prove this by induction on $(n,m) \in (\mathbb{N} \times \mathbb{N}) \backslash (\mathbb{N} \times \{0\})$ following the lexicographic order. Fix $(n,m) \in  (\mathbb{N} \times \mathbb{N}) \backslash (\mathbb{N} \times \{0\})$ and assume that the property $P_{k,l}$ is true for all $(k,l) <_{\texttt{lex}} (m,n)$. Fix $X \in \mathcal{S}_m^d$ and $Z \in \mathcal{S}_n^d$ such that $X \cap Z = \emptyset$. There exists a sequence $(Z_r)_r$ of shifts in $\mathcal{S}_{n-1}^d$ such that $Z_r \rightarrow Z$ and for all $r$, $Z_r \neq Z$. \textbf{Case $\boldsymbol{m = 0}$.} For all $r$, $X \cup Z_r$ is in $\mathcal{S}_{n-1}^d$: if $n=1$ this is trivial, otherwise this comes from the induction hypothesis. Since $X \cup Z_r \rightarrow X \cup Z$ and for all $r$ we have $X \cup Z_r \neq X \cup Z$, this implies that $X \cup Z \in \mathcal{S}_n^d$. \textbf{Case $\boldsymbol{m > 0}$.} In this case, consider any sequence $(X_r)$ in $\mathcal{S}_0^d$ which converges towards $X$. Then by induction, for all $r$ we have $X_r \cup Z \in \mathcal{S}_n^d$. Since $X_r \cup Z \rightarrow X \cup Z$ and for all $r$, $X_r \cup Z \neq X \cup Z$, this implies that $X \cup Z \in \mathcal{S}_{n+1}^d$.
\end{proof}

\begin{lemma}\label{lemma:neighborhood}
    Let $X,Z$ be two shifts of finite type whose subsystems are all of finite type, and let $U$ (resp. $V$) be a neighborhood of $X$ (resp. $Z$) which consists  of subsystems of $X$ (resp. $Z$). The set $W$ of shifts of the form $M \cup N$, where $M \in U$ and $N \in V$, is a neighborhood of $X \cup Z$.
\end{lemma}

\begin{proof}
Consider a shift of the form $M \cup N$, where $M \in U$ and $N \in V$. Since $M \in U$ and $N \in V$, they are subsystems of $X$ and $Z$, respectively, and thus by hypothesis they are of finite type. Since $X$ and $Z$ are disjoint, $M$ and $N$ are also disjoint. Therefore, $M \cup N$ is a shift of finite type. 
Set $\epsilon_0 < \min_{x \in M} \min_{z \in N} \delta (x,z)$ such that every shift which is at distance less than $\epsilon_0$ from $M \cup N$ is a subsystem of $M \cup N$ and every shift at distance less than $\epsilon_0$ from $M$ (resp. $N$) is in $U$ (resp. $V$). Consider a shift $Y$ such that  $\delta_H(Y,M \cup N) < \epsilon_0$. This implies that $Y$ is a subsystem of $M \cup N$ and can be written as $Y = (Y \cap M) \cup (Y \cap N)$. Both $Y \cap M$ and $Y \cap N$ are shifts and since $\delta_H(Y,M \cup N) < \epsilon_0$ and $\epsilon_0 < \min_{x \in M} \min_{z \in N} \delta (x,z)$, they are at distance less than $\epsilon_0$ from $M$ and $N$ respectively. Indeed, consider $x \in M$. Since $\delta_H(Y,M \cup N) < \epsilon_0$, there exists $y \in Y$ such that $\delta(x,y) \le \epsilon_0$. Since $\epsilon_0 < \min_{x \in M} \min_{z \in N} \delta (x,z)$, we have $y \in Y \cap M$. Since $Y \cap M \subset M$, this is enough to conclude that $\delta_H(Y\cap M , M) \le \epsilon_0$. Similarly we have $\delta_H(Y\cap N , N) \le \epsilon_0$.
Thus, $Y\cap M$ is in $U$ and $Y\cap N$ is in $V$. This implies that $Y$ is in $W$.
  
\end{proof}

\begin{lemma}\label{lemma.hierarchy}
    For all $n \ge 1$, $0 \le m \le n$, we have the following property denoted by $Q_{n,m}$. For every $X \in \mathcal{D}_{m}^{d}$ and $Z \in \mathcal{D}_{n}^{d}$ such that $X \cap Z = \emptyset$, if $m =0$ the shift $X \cup Z$ is in $\mathcal{D}^d_{n}$, otherwise it is in $\mathcal{D}^d_{n+1}$.
\end{lemma}

\begin{proof}
Again we prove this by induction on $(n,m) \in  (\mathbb{N} \times \mathbb{N}) \backslash (\mathbb{N} \times \{0\})$ following the lexicographic order. Fix $(n,m) \in  (\mathbb{N} \times \mathbb{N})\backslash (\mathbb{N} \times \{0\})$ and assume that the property $Q_{k,l}$ is true for all $(k,l) <_{\texttt{lex}} (n,m)$. Take $X \in \mathcal{D}_m^d$ and $Z \in \mathcal{D}_n^d$ such that $X \cap Z$ is empty. Observe that the subsystems of $X\cup Z$ are its non-empty subsets of the form $M \cup N$, where $M$ (resp. $N$) is empty or a subsystem of $X$ (resp. $Z$).
Since the disjoint union of two shifts of finite type is of finite type, all the subsystems of $X \cup Z$ are of finite type.
\textbf{Case $\boldsymbol{m = 0}$.} The shift $X \cup Z$ is in $\mathcal{S}_{n}^d$ by Lemma \ref{lemma.recur.s}.
We have left to prove that $X \cup Z$ is isolated in $\mathcal{S}_{{n}}^d$, meaning that it has a neighborhood $W$ such that $W\backslash \{X\cup Z\} \subset \mathcal{C}_{{n-1}}^d$. 
% = \mathcal{I}^d$.
Since $Z$ is of finite type and isolated in $\mathcal{S}_{{n}}^d$, there exists a neighborhood $U$ of $Z$ which consists of subsystems of $Z$ which are in $\mathcal{C}_{{n-1}}^d$. Denote by $W$ the set of 
subsystems of $X \cup Z$ of the form $X \cup N$, where $N \in {U}$. Since $X$ is isolated in $\mathcal{S}_0^d$, {by Lemma \ref{lemma:neighborhood}}, the set $W$ is a neighborhood of $X \cup Z$. Any element of $W$ different from $X \cup Z$ is the disjoint union of {a shift is $\mathcal{C}_{n-1}^d$ and a shift in $\mathcal{C}_{0}^d$. For $n> 1$, this implies that every element of $W$ different from $X \cup Z$ is in $\mathcal{C}_{n-1}^d$. In the case $n = 1$, every element of $W$ different from $X \cup Z$ is the union of two shifts in $\mathcal{I}^d$, which is in $\mathcal{I}^d$.}
 {Since the subsystems of both of these shifts are all of finite type and that the disjoint union of two shifts of finite type is of finite type, every element of $W$ is in $\mathcal{C}_0^d$.}
\textbf{Case $\boldsymbol{m > 0}$.} 
The shift $X \cup Z$ is in $\mathcal{S}_{n+1}^d$ by Lemma \ref{lemma.recur.s}. We have left to prove that $X \cup Z$ is isolated in $\mathcal{S}_{n+1}^d$, meaning that it has a neighborhood $W$ such that $W\backslash \{X\cup Z\} \subset \mathcal{C}_n^d$.
Since $X$ (resp. $Z$) is of finite type and isolated in $\mathcal{S}_{m}^d$ (resp. $\mathcal{S}_{n}^d$), there exists a neighborhood $U$ (resp. $V$) of $X$ (resp. $Z$) such that $U \backslash \{X\}$ (resp. $V \backslash \{Z\}$) consists of subsystems of $X$ (resp. $Z$) which are in $\mathcal{C}_{m-1}^d$ (resp. $\mathcal{C}_{n-1}^d$). Denote by $W$ the set of 
subsystems of $X \cup Z$ of the form $M \cup N$, where $M \in U$ and $N \in V$. {By Lemma \ref{lemma:neighborhood},} the set $W$ is a neighborhood of $X \cup Z$. Any element of $W$ different from $X \cup Z$ is the disjoint union of a shift in $\mathcal{D}_k^d$ with a shift in $\mathcal{D}_l^d$, where $k < m$ and $l  {\le}  n$, {or $k \le  m$ and $l < n$}. {In any case,} by induction, we can deduce that any element of $W$ different from $X \cup Z$ is in $\mathcal{C}_n^d$.
% In order to prove that $X \cup Z$ in $\mathcal{D}_2^d$, it is sufficient observe that all its strict subsystems are of the form $M \cup N$, where  
% following forms: i) $N\cup M$; ii) $N\cup Z$; iii) $X\cup M$; iv) $X$; v)$Z$; where $N\subsetneq X$ and $M\subsetneq Z$. Since $N$ and $M$ are isolated, by \cref{lemma.union.2}, subsystems of the form i) are isolated. Subsystems of the form iv) and v) are in $\mathcal{D}_1^d$ by hypothesis. Subsystems of the form ii) are in $\mathcal{D}_1^d$, since any of their subsystems of the form i) o just $N$. A symmetric argument shows that subsystems of the form iii) are in $\mathcal{D}_1^d$ as well. This means that $X\cup Z$ is in $\mathcal{D}_2^d$.
% An analogous argument shows that if the claim holds for $n=k$, then it holds for $n=k+1$.
\end{proof}

% \begin{proof}
%     Let us prove this by recursion on $n$. Set $n=0$. By Lemma~\ref{lemma.union.2}, the subsystems of $X \cup Z$ which are union of two isolated shifts are also isolated. The other subsystems of $X \cup Z$ are of the form $X \cup M$
%     or $N \cup Z$, where $N$ and $M$ are subsystems of $X$ and $Z$, respectively.
%     The subsystems of these subsystems are, thus, unions of isolated points and are isolated. 
%     Therefore all subsystems of $X \cup Z$ are isolated or in $\mathcal{D}^d_{0}$. This implies that $X \cup Z$ is in $\mathcal{D}^d_{1}$. An analogous argument shows that if the claim holds for $k$, then it holds for $k+1$.
% \end{proof}
% \comm{Silvere}{In fact there is a problem in this result. When taking subsystems of a type iii) subsystem, it can still be a subsystem of type iii) and not isolated necessarily. I believe the precise formulation of all these intermediate results can be changed in order to get the main results there (infinite rank) but it needs just a bit of work.}
% \comm{Silvere}{Replace condition that all subsystems are in the previous set by the existence of a neighborhood whose elements are all in the previous set ? In fact this is the def cantor-bendixson sets so no need for new notations.}

For any $d$-dimensional shift $X$ and any integer $n \ge 0$, we denote by $\mathcal{G}_n(X)$ the shift $X \times Z_n$, where $Z_n$ is the $d$-dimensional shift on alphabet $\llbracket 1, n \rrbracket$ whose elements are the constant configurations with value in $\llbracket 1, n \rrbracket$.

\begin{theorem}[\cref{thm:C}]\label{thm.Cantor}
The Cantor-Bendixson rank of $\mathcal{S}^d$, $d>1$, is infinite.
\end{theorem}

% \comm{Silvere}{In light of the changes above; I think in the following proof we should replace $\mathcal{D}_0^{d+2}$ with $\mathcal{D}_1^{d+2}$ and probably $\mathcal{D}_{n-1}^{d+2}$ with $\mathcal{D}_n^{d+2}$}

% \comm{A}{or maybe just replace $n\geq1$ by $n\geq2$, after the reference to Lemma 5.18. So that the ``first case'' is $\mathcal{G}_2(\mathcal{E}^d(X_0))\in\mathcal{D}_1^{d+2}$ and we're safe.}

% \comm{S}{That works}

\begin{proof}
We already know from Theorem~\ref{theorem.times.shift} that $\mathcal{E}^d(X_0) \in \mathcal{D}^{d+2}_{{1}}$ for all $d \ge 0$. 
By Lemma~\ref{lemma.hierarchy}, for all $n \ge 1$, 
$\mathcal{G}_n(\mathcal{E}^d(X_0)) \in \mathcal{D}^{d+2}_{{n}}$. Thus for all $n$, the set {of isolated points in} $\mathcal{S}^d_{n}$ is not empty. This implies that the Cantor-Bendixson rank of $\mathcal{S}^d$ is infinite. 
\end{proof}
%\textcolor{cyan}{The claim in the previous theorem is exactly the claim in \cref{thm:C} from the introduction which is now proven.}
\begin{comment}
Produits de x2x3 : on a un rang de cantor bendixon supérieur à omega2. Selon le me principe on aurait omega^n en dim n.
En prenant des unions finies de copies de x2x3 on peut voir que le rang de cantor bendixon est au moins Omega.
En faisant un produit de x2x3 avec Robinson on réalise Omega, et avec plusieurs produits on a que ce rang est au moins Omega^2.
\end{comment}

\subsubsection{Type $\infty$ shifts are in the Cantor-Bendixson residue}\label{section.cb.residue}

\begin{proposition}
    The set of maximality type $\infty$ shifts is included in the Cantor-Bendixson residue. Furthermore, it has empty interior.
\end{proposition}

\begin{proof}
    Consider any shift $X$ which has maximality type $\infty$. Denote by $T_k, k \in \mathbb{N}$, the elements of $\mathcal{T}(X)$. \begin{enumerate}
        \item Using Theorem~\ref{thm:decomposition}, for all $k \in \mathbb{N}$, 
    the set $\mathcal{K}(X) \cup \left(\bigcup_{n \neq k} T_n \right)$ is also of maximality type $\infty$, and this sequence converges towards $X$ by Lemma~\ref{lemma.max.dec}. As a consequence, the set of maximality type $\infty$ shifts has no isolated point. This implies that the closure of this set is contained in the Cantor-Bendixson residue.
    \item Again by Theorem~\ref{thm:decomposition}, for all $n$, the set $\mathcal{K}(X) \cup \left(\bigcup_{k=0}^n T_k\right)$ has maximality type $n+1$ and this sequence converges to $X$. Thus $X$ is the limit of a sequence of finite maximality type shifts. Since this is true for all $X$, the set of maximality type $\infty$ shifts has empty interior.
    \end{enumerate}
\end{proof}

\begin{remark}
A direct consequence is that the closure of this set is also in the Cantor-Bendixson residue.
\end{remark}

% \begin{remark}
    % It would be natural to think 
    % that the sets $K \cup \left( \bigcup_{n \le k} T_n\right)$ are of type $n$ and since they converge to $X$ when $k \rightarrow \infty$, we can decude that the set of maximality type $\infty$ shifts has empty interior. But they are not necessarily of type $n$, because $K$ is not necessarily empty or of maximality type $0$. Hence the following question: does this set have empty interior?
% \end{remark}

\section{Open questions\label{open.questions}}

We conclude our discussion with some open questions. 

\paragraph*{Main  questions} The most relevant questions are the following: 

\begin{question}
    What is the closure of the set of isolated shifts in $\mathcal{S}^d$?
\end{question}

We expect that the systems in this closure satisfy some conditions in terms of the structure of the subsystems set, however these conditions are not 
apparent. Another important question is: 

\begin{question}
    What is the Cantor-Bendixson rank of $\mathcal{S}^d$?
\end{question}

% We expect that other constructions based on the shift $X_0$ to lead to other lower bounds on this rank. It is not clear though if this will be sufficient. \bigskip

We conjecture the following: 

\begin{conjecture}
    The Cantor-Bendixson rank of the space of $d$-dimensional shifts is $\omega$ for all $d > 1$. Furtheremore, the Cantor-Bendixson residue is the closure of the set of shifts having maximality type infinity. The shifts which are not in this residue are the ones which can be written as a finite union of shifts of finite type which are the limit of their minimal subsystems.
\end{conjecture}

Another strategy in order to understand the topological structure of $\mathcal{S}^d$ is to begin with understanding the structure of the subsystems of shifts of finite type outside of full shifts. In particular: 

\begin{question}
    What are the shifts of finite type in $\mathcal{S}^d$ whose set of (transitive)
    %\textcolor{blue}{which set? (transitive)?} 
    subsystems is a perfect set?
\end{question}

\paragraph*{Relations between the notions introduced}

The main outcome of our study is a battery of concepts which may be studied independently. A better understanding of these concepts 
could be useful in order to understand genericity for higher dimensional shifts: the notion of maximal subsystem, and the maximality type; certain properties of shifts in relation with isolated shifts (having only isolated subsystems or no isolated subsystems at all), or the shifts which have a neighborhood in which each element contains it, the shifts which satisfy the $(\star)$ property. 
One could, for instance, wonder about the possible maximality types for systems in each of the classes described, or the closure and interior of these classes. Characterize systems whose subsystems are all in a certain class, or whose neighborhood has only systems in a given class. Which of these classes have non empty intersection?

\paragraph*{Properties of isolated shifts} Questions which could lead to a better understanding of isolated points of $\mathcal{S}^d$ are of interest. In particular, we suspect that isolated points have low complexity: 

\begin{question}
    Do isolated points in $\mathcal{T}^d$ have zero entropy?
\end{question}

Another way is to explore the limit structures of subsystems in the set of isolated shifts. For instance:

\begin{question}
Is there a decreasing infinite sequence of shifts $(X_n)_{n\ge0}$ such that
$X_0$ is of finite type and for all $n$, $X_{n+1}$ is the unique maximal subsystem of $X_n$? If so, what are the properties of $\bigcap_n X_n$?
\end{question}

\paragraph*{Transitivity and maximality type}

We have proved that a transitive shift has maximality type $0$ or $1$. On the other hand, we know that not all shifts of maximality type $0$ or $1$ are transitive. Is it possible to distinguish how non-transitivity of shifts with maximality type $\ge 2$ differs from the non-transitivity of those which have maximality type $0$ or $1$?

\end{document}